\documentclass{amsart}
\usepackage[T1]{fontenc}
\usepackage{amssymb, amsmath, color, textcomp, url,tikz, 
mathtools}
\usepackage{bbm}

\usetikzlibrary{matrix,arrows}
\usetikzlibrary{fit}
\usetikzlibrary{positioning}
\usepackage[labelformat=empty]{caption}

\usepackage[normalem]{ulem}
\usepackage{soul}
\usepackage{mdwlist}

\usepackage{enumerate,xspace}
\usepackage{chngcntr, csquotes} 
\theoremstyle{plain}
\newtheorem{theorem}{Theorem}[section]
\newtheorem{cor}[theorem]{Corollary}
\newtheorem{prop}[theorem]{Proposition}
\newtheorem{lemma}[theorem]{Lemma}

\newcounter{proofcount}
\AtBeginEnvironment{proof}{\stepcounter{proofcount}}

\newtheorem*{question}{Question}
\newtheorem*{claim*}{Claim}
\makeatletter                  
\@addtoreset{claim}{proofcount}
\makeatother                   

\newenvironment{claimproof*}[1][Proof of Claim.] 
{%
	\proof[#1]%
	
}
{%
	\endproof%
}

\newtheorem*{teo}{Main Theorem}

\theoremstyle{definition}
\newtheorem{remark}[theorem]{Remark}
\newtheorem{fact}[theorem]{Fact}
\newtheorem{definition}[theorem]{Definition}
\newtheorem{example}[theorem]{Example}
\newtheorem{hyp}{Assumption}

\newcommand{\nc}{\newcommand}

\nc{\Z}{\mathbb{Z}}
\nc{\Q}{\mathbb{Q}}
\nc{\N}{\mathbb{N}}
\nc{\F}{\mathbb{F}}
\nc{\UU}{\mathbb{U}}
\nc{\C}{\mathbb{C}}

\nc{\M}{\mathcal{M}}
\nc\LL{\mathcal L}
\nc\II{\mathcal I}
\nc{\Ring}{\mathcal R}
\nc\BB{\mathcal B_M^Y}
\nc{\RB}{\mathrm{R}(\BB)}

\nc{\stt}{\operatorname{St}}
\nc{\stab}{\operatorname{Stab}}
\nc{\GO}[1]{G_{#1}^{00}}
\nc{\sbgp}[1]{\langle\xspace {#1}\xspace\rangle}
\nc{\Conn}[1]{\langle\xspace {X}\xspace\rangle^{00}_{#1}}
\nc{\band}[1]{\bar d_{\mathcal{#1}}}
\nc{\Cos}[1]{\operatorname{Cos}(#1)}

\nc\Def{\operatorname{Def}}

\nc{\dcl}{\operatorname{dcl}}

\nc{\acl}{\operatorname{acl}}

\nc{\nf}[1]{_{\mid {#1}}}
\nc{\restr}[1]{\xspace_{\upharpoonright {#1}}}

\nc\inv{ ^{-1}}
\nc\Coh{\operatorname{Coh}}

\nc{\supp}{\operatorname{\mathrm{supp}}}

\nc{\tp}{\operatorname{tp}}
\nc\cb{\operatorname{Cb}}
\nc\U{\operatorname{U}}
\nc{\cf}{\text{cf.\,}}
\nc{\eg}{\text{e.g. }}

\def\Ind#1#2{#1\setbox0=\hbox{$#1x$}\kern\wd0\hbox to
  0pt{\hss$#1\mid$\hss} \lower.9\ht0\hbox to
  0pt{\hss$#1\smile$\hss}\kern\wd0}
\def\Notind#1#2{#1\setbox0=\hbox{$#1x$}\kern\wd0\hbox to
  0pt{\mathchardef\nn="0236\hss$#1\nn$\kern1.4\wd0\hss}\hbox to
  0pt{\hss$#1\mid$\hss}\lower.9\ht0 \hbox to
  0pt{\hss$#1\smile$\hss}\kern\wd0}

\def\indip{\mathop{\ \ \hbox to 0pt{\hss$\mid^{\hbox to
0pt{$\scriptstyle P$\hss}}$\hss}
\lower4pt\hbox to 0pt{\hss$\smile$\hss}\ \ }}
\def\nindip{\mathop{\ \ \hbox to 0pt{\hss$\!\not{\mid}^{\hbox to
0pt{$\scriptstyle\, P$\hss}}$\hss}
\lower4pt\hbox to 0pt{\hss$\smile$\hss}\ \ }}

\begin{document}

\title[Complete Amalgamation]{Complete type amalgamation for non-standard finite groups}
\date{\today}

\author{Amador Martin-Pizarro and Daniel Palac\'in}
\address{Abteilung f\"ur Mathematische Logik, Mathematisches Institut,
  Albert-Ludwigs-Universit\"at Freiburg, Ernst-Zermelo-Stra\ss e 1, 
  D-79104
  Freiburg, Germany}
\address{Departamento de \'Algebra, Facultad de Matem\'aticas,
	Universidad Complutense de Madrid, 28040 Madrid, Spain}

\email{pizarro@math.uni-freiburg.de}
\email{dpalacin@ucm.es}

\thanks{Research supported by MTM2017-86777-P as well as by the Deutsche
	Forschungsgemeinschaft (DFG, German Research Foundation) - 
	Project number 2100310201 and 2100310301, part of the ANR-DFG 
program GeoMod}

\begin{abstract}
We extend previous work on Hrushovski's stabilizer's theorem and 
prove a measure-theoretic version of a well-known result of 
Pillay-Scanlon-Wagner on products of  three types. This generalizes 
results of Gowers  on products of three sets and yields 
model-theoretic proofs of existing asymptotic results for 
quasirandom 
groups. We also obtain a model-theoretic proof of Roth's theorem on the existence of arithmetic 
progressions of length $3$ for subsets of positive density in suitable definably amenable groups, such as countable amenable abelian groups without involutions and ultraproducts of finite abelian groups of odd order.
 \end{abstract}

\maketitle

\section*{Introduction}

Szemerédi answered positively a question of  Erd\H{o}s and Tur\'an 
by showing \cite{eS75} that every subset $A$ of $\N$ of  upper 
density 
\[ \limsup\limits_{n\to \infty} \frac{\left|A\cap 
\{1, \ldots, n\}\right|}{n}>0\] must 
contain an arithmetic progression of length $k$ for every natural 
number $k$. For $k=3$, the existence of arithmetic progressions of 
length $3$ (in short $3$-AP)  was already proven by Roth in what is 
now called Roth's theorem on arithmetic progressions \cite{kR53} 
(not to be confused with Roth's theorem on  diophantine 
approxi\-mation of algebraic integers). There has been 
(and still is) impressive work done on understanding Roth's and 
Szemerédi's theorems, explicitly computing lower bounds for the 
density as well as extending these results to more general 
settings.  In the second direction, it is worth mentioning Green and 
Tao's result on the existence of arbitrarily long finite arithmetic 
progressions among the subset of prime numbers \cite{GT08}, 
which however has upper density $0$. 

In the non-commutative setting, proving single instances of 
Szemerédi's 
theorem, particularly Roth's theorem, becomes highly 
non-trivial. Note that the sequence 
$(a,ab, ab^2)$ can be seen as a $3$-AP, even for  non-commutative 
groups.  Gowers asked 
\cite[Question 6.5]{wG08} whether the proportion of pairs $(a,b)$ 
in $\mathrm{PSL}_2(q)$, for $q$ a prime power, 
such that $a$, $ab$ and $ab^2$ all lie in a fixed subset $A$ of density 
$\delta$ approximately equals $\delta^3$. 
Gower's question was positively answered 
by Tao \cite{tT13} and later extended to arbitrary 
non-abelian finite simple groups by Peluse \cite{sP18}. For 
 arithmetic 
progressions $(a, ab, ab^2, ab^3)$ of length $4$ in 
$\mathrm{PSL}_2(q)$, a partial result was  obtained in 
\cite{tT13}, whenever the element $b$ is diagonalizable 
over the finite field $\mathbb F_q$ (which happens half of the time). 

A different  generalization of Roth's theorem, 
present in work of Sanders \cite{tS09} and Henriot  \cite{kH16},    
concerns the existence of  a $3$-AP in finite sets of small doubling in abelian groups. 
Recall that a finite set $A$ of a group has doubling at most 
$K$  if the productset $A\cdot A=\{ab\}_{a, b\in A}$ has 
cardinality 
$|A\cdot A|\le K|A|$. More generally, a finite set has tripling at 
most  $K$ if $|A\cdot A\cdot A|\le K|A|$.  If $A$ has tripling at 
most $K$, the comparable set $A\cup A\inv \cup\{\mathrm{id}_G\}$ (of size at 
most $2|A|+1$) has tripling at most 
$(CK^C )^2 $ with respect to some explicit absolute
constant $C>0$, so we may assume that $A$ is symmetric 
and contains the neutral element. Archetypal sets of small doubling 
are approximate subgroups, that is, symmetric sets $A$ such that 
$A\cdot A$ is covered by finitely many translates of $A$. 
The model-theoretic study of  approximate 
subgroups  first appeared in 
Hrushovski's striking paper \cite{eH12}, which contained the 
so-called  stabilizer theorem, adapting techniques from stability 
theory to an abstract measure-theoretic setting. Hrushovski's work 
has led to several remarkable applications  of model theory 
to additive combinatorics. 

In classical geometric model theory, and more generally, in a group 
$G$ 
definable in a simple theory, Hrushovski's stabilizer of a generic 
type over an elementary substructure $M$ is the connected 
component $\GO M$, that is, the 
smallest type-definable subgroup over $M$ of bounded index 
(bounded with respect to the saturation of the ambient universal 
model). Generic types in $\GO M$ are called $\emph{principal 
types}$. If the theory is stable, there is a unique principal type, but 
this need not be the case for simple theories. However, 
Pillay, Scanlon and Wagner noticed in \cite[Proposition 
2.2]{PSW98} that for every three principal types $p$, $q$ and $r$ in 
a simple theory over an elementary substructure $M$, there are 
independent realizations $a$ of $p$ and $b$ of $q$ over $M$ such 
that $a\cdot b$ realizes $r$. The main ingredient in their proof  is a clever 
application of  $3$-complete amalgamation (also known as 
\emph{the 
independence theorem}) over the elementary  substructure $M$. For 
the purpose of the present work, we shall not define what a general 
complete amalgamation problem is, but a variation of it, restricting the 
problem to  conditions 
given by products with respect to the underlying group  law: 
\begin{question}
Fix a natural number $n\geq 2$.  For each non-empty subset $F$ of 
$\{1,\ldots, 
n\}$, let $p_F$ be a principal generic (that is, weakly random) type  over the  elementary substructure $M$. Can we find (under suitable conditions) an 
independent (weakly random) tuple $(a_1,\ldots, a_n)$ of 
$G^n$ such that 
for all $\emptyset\neq F\subseteq\{1,\ldots, n\}$, the element 
$a_F$ realizes $p_F$, where $a_F$ stands for the product of all 
$a_i$, with $i$ in $F$, written with the indices in increasing order? 
\end{question}
The above formulation resonates with \cite[Theorem 
5.3]{GT08} for quasirandom groups and agrees for $n=2$  with the 
aforementioned result of Pillay, Scanlon and Wagner. 

In this work, we will give a (partial)
positive solution for $n=2$ (Theorem \ref{T:pqr}) to the above 
question for definable groups equipped with a definable Keisler measure satisfying Fubini (e.g. ultraproducts of groups equipped with the 
associated counting measure localized with respect to a  
distinguished finite set, as in Example \ref{E:amenable}).  As a 
by-product, we obtain a measure-theoretic version 
of the result of Pillay, Scanlon and Wagner (Theorem \ref{T:pqr}):
\begin{teo}\label{T:Teo_pqr}
Given a pseudo-finite subset $X$ of small tripling in a sufficiently saturated group $G$ and a countable elementary substructure $M$, for every weakly random type $q$ and almost all pairs $(p,r)$ of weakly random types over $M$  concentrated in the subgroup $\sbgp X$ generated by $X$, there is a weakly random pair $(a, b)$ over $M$ in $p\times q$ with $a\cdot b$ realizing $r$, whenever $\Cos{p}\cdot \Cos{q}=\Cos{r}$, where $\Cos{p}$ is the coset of $\Conn M$ determined by the type $p$. 
\end{teo}
The result of Pillay, Scanlon and Wagner holds for all such pairs 
$(p,r)$ of generic types. Unfortunately, our techniques can only 
prove the analogous result outside a set of measure $0$. Whilst we 
do not know how to obtain the result for all pairs $(p,r)$ of weakly 
random types over $M$, 
our results however suffice to reprove model-theoretically some 
known 
results. Using a model-theoretic analog of Croot-Sisask's almost periodicity \cite[Corollary 1.2]{CS10} (Corollary \ref{C:measure_constant_coset}), we easily deduce a non-quantitative version of Roth’s theorem (Theorem \ref{T:Roth_defamen}) 
on 3-AP for finite subsets of small doubling in abelian groups with trivial 2-torsion, which resembles previous work of Sanders \cite[Theorem 7.1]{tS09} and
generalizes a result of Frankl, Graham and Rödl \cite[Theorem 1]{FGR87}.

In Section \ref{S:ultra}, we reprove model-theoretically results valid for ultra-quasirandom  groups, that is, asymptotic limits of quasirandom groups, already studied by Bergelson and Tao 
\cite{BT14}, and later by the second  author \cite{dP20}. In 
particular, in Corollary \ref{C:gowers_mixing} we give 
non-quantitative model-theoretic proofs of Gower's results 
\cite[Theorem 	3.3\ \& Theorem 5.3]{wG08}. In Section 
\ref{S:localultra}, we explore further this analogy to extend some of 
the results of Gowers to a local setting, without imposing that the 
group is an ultraproduct of quasirandom groups (see Corollaries 
\ref{C:ppaln=2} and \ref{C:ppaln_ge3}).

We will assume throughout the text a certain familiarity with basic notions in model theory. Sections $\ref{S:Fubini}, \ref{S:stable}$ and $\ref{S:pqr}$ contain the model-theoretic core of the paper, whilst Sections \ref{S:ultra} and \ref{S:localultra} contain applications to additive combinatorics.  

\subsection*{Acknowledgements} We are most indebted to  Ehud Hrushovski and Julia Wolf for their helpful remarks which have considerably improved this article.  We would like to express our gratitude to Angus Matthews for pointing out a mistake in a previous version of Theorem \ref{T:pqr}. 

We would also like to express our sincere gratitude to the anonymous referee of the previous versions of this article for the insightful comments and remarks, as well as for providing a clever way to circumvent some additional assumptions in a previous proof of Roth's Theorem (Theorem \ref{T:Roth_defamen}). The statement and the proof of Lemma \ref{L:ref} are solely due to the referee. We are most thankful that we are allowed to include the proof here.  

\section{Randomness and Fubini}\label{S:Fubini}

Most of the material in this section can be found in \cite{Halmos, eH12, 
MW15, pSbook}.

We  work inside a sufficiently saturated model $\mathbb U$ of a complete 
first-order theory (with infinite models) in a countable language $\LL$, that is, 
the model $\UU$ is saturated and strongly homogeneous with 
respect to some  sufficiently large  cardinal $\kappa$. 
All sets  and tuples are taken inside $\mathbb U$. 

A subset  $X$ of $\UU^n$ is definable over the parameter set $A$ 
if there exists a formula $\phi(x_1,\ldots,x_n, y_1, \ldots, y_m )$ 
and a tuple $a=(a_1,\ldots, a_m)$ in $A$ such that an $n$-tuple $b$ 
belongs to $X$ if and only if  $\phi(b, a)$ holds in $\UU$. As usual, 
we identify a definable subset of $\UU$ with a formula defining it. 
Unless explicitly stated, when we use the word 
definable, we mean 
definably possibly with parameters. 
It follows that a subset $X$ is definable over the parameter set $A$ 
if and only if $X$ is definable (over some set of 
parameters)  and invariant under 
the action of the group of automorphisms $\mathrm{Aut}(\UU/A)$ 
of $\UU$ fixing $A$ pointwise.  The subset $X$ of $\UU$ is 
type-definable if it is the intersection of a bounded number of 
definable sets, where bounded means that its 
size is strictly smaller than the degree of  saturation of  $\UU$. 

For our applications we will mainly consider the case where the 
language $\LL$ contains the language of groups and the universe of 
our ambient model is a group. Nonetheless, our model-theoretic 
setting works as well for an arbitrary definable group, that is, a group 
whose underlying set and its group law are both definable.  

\begin{definition}\label{D:amenable}
A \emph{definably amenable pair} $(G, X)$ consists of an underlying definable 
group $G$ together with the following data:

\begin{itemize}
    \item A definable subset $X$ of $G$; 
    \item The (boolean) ring $\Ring$ of definable sets contained in the subgroup $\sbgp  X$ generated by $X$, that is,  the subcollection $\Ring$ is closed under finite unions and relative set-theoretic differences;
    \item A finitely additive measure $\mu$ on $\Ring$ invariant 
    under both left and right translation with $\mu(X)=1$. (Note that 
    we require translation invariance under both actions). 
\end{itemize}
\end{definition}

Note that the subgroup $\sbgp X$ generated by the subset $X$ need not be definable, but it is 
\emph{locally 
definable}, for the subgroup $\sbgp X$ is a countable union of 
definable sets of the form 
\[ X^{\odot n}= \underbrace{X_1\cdots X_1}_{n},\]where $X_1$ is the definable 
set $X\cup 
X\inv\cup \{\mathrm{id}_G\}$. 
Furthermore, every definable subset $Y$ of $\sbgp X$ is 
contained in some finite product $X^{\odot
n}$, by compactness and saturation of the ambient model. 

\begin{remark}\label{R:Cara}
Model-theoretic compactness implies that the finitely additive 
measure $\mu$ satisfies  Carath\'eodory's criterion, so there exists a 
unique $\sigma$-additive measure on the $\sigma$-algebra 
generated by $\Ring$. On the other hand, for every definable set $Y$ 
of $\Ring$ over any set of parameters $C$, the measure $\mu$ 
extends to a regular Borel finite measure on the Stone space 
$S_Y(C)$ of complete types over $C$ containing the $C$-definable 
set $Y$, see \cite[p. 99]{pSbook}. 
\end{remark}

We will denote the above extension of $\mu$  again by $\mu$, though there will be (most likely) Borel sets of  infinite measure, 
as noticed by Massicot and Wagner: 

\begin{fact}\textup{(}\cite[Remark 4]{MW15}\textup{)} 
The subgroup $\sbgp X$ is definable if and only if $\mu(\sbgp X)$ 
is finite. 	
\end{fact}

	Throughout the paper, we will always assume that the language 
	$\LL$ is rich enough (see \cite[Definition 3.19]{sS17}) to render the measure 
$\mu$ definable without parameters.

\begin{definition}\label{D:meas_def}
The measure $\mu$ of a  definably amenable pair $(G,X)$  is 
\emph{definable without parameters} if for every
$\LL$-formula $\varphi(x,y)$, every natural number $n\ge 1$ and 
every $\epsilon>0$, there is a partition of the $\LL$-definable set 
\[\{y \in \UU^{|y|}  \ | \ \varphi(\UU,y) \subseteq X^{\odot n}\}\] 
into $\LL$-formulae $\rho_1(y),\ldots, \rho_m(y)$ such that 
whenever a
pair $(b, b')$  in $\UU^{|y|}\times \UU^{|y|}$ realizes $\rho_i(y)\land\rho_i(y')$, then 
\[ |\mu(\varphi(x,b)) -\mu(\varphi(x,b') )|<\epsilon. \] 
\end{definition}  

The above definition is  a mere formulation of \cite[Definition 
3.19]{sS17} to the locally definable context, by imposing that the 
restriction of $\mu$ to every  definable subset $X^{\odot n}$ is 
definable in the sense of  \cite[Definition 3.19]{sS17}. In particular, 
a definable measure of  a definably amenable pair $(G,X)$ is 
\emph{invariant},  that is, its value is invariant under the action of 
$\mathrm{Aut}(\UU)$. 
Notice that whenever the measure $\mu$ is definable, given a definable subset  $\varphi(x,b)$ of measure $r$ and a value $\epsilon>0$, the tuple $b$ lies in some definable subset which is contained in 
\[
\left\{y\in \UU^{|y|} \ | \ r-\epsilon\le \mu(\varphi(\UU,y)) \le r+\epsilon \right\}.
\] 
Assuming that $\mu$ is definable, its extension to the $\sigma$-algebra generated by the 
definable subsets 
of  $\sbgp X$ is again invariant under left and right translations, as well as 
under automorphisms:  Indeed,  every 
automorphism $\tau$ of $\mathrm{Aut}(\UU)$ (likewise for left and right translations) gives rise to 
a measure $\mu^\tau$, such that  $\mu^\tau(Y) = 
\mu(\tau(Y))$ for every measurable subset $Y$ of 
$\sbgp X$.  Since $\mu^\tau$ 
agrees with $\mu$ on $\Ring$, we 
conclude that the $\sigma$-additive measure $\mu^\tau=\mu$ by the uniqueness of the 
extension. Thus, the measure of a Borel subset $Y$ in the space of 
types containing a  fixed definable set $Z$ in $\Ring$ depends solely on the type of the  
parameters defining $Y$. 

\begin{example}\label{E:amenable}
Let $(G_n)_{n\in \mathbb N}$ be an infinite family of groups, each 
with a distinguished finite subset $X_n$. Expand the language of  
groups to a language $\LL$ including a unary predicate and set 
$M_n$ to be an $\LL$-structure with universe $G_n$, equipped 
with its group operation, and interpret the predicate as $X_n$. 
Following \cite[Section 2.6]{eH12} we can further assume that 
$\LL$ has predicates $Q_{r,\varphi} (y)$ for each 
$r$ in $\mathbb Q^{\geq 0}$ and every formula  $\varphi(x,y)$ in 
$\LL$ such that $Q_{r,\varphi}(b)$ holds if and only if  the set 
$\varphi(M_n,b)$ is finite with $|\varphi(M_n,b)| \le r |X_n|$. 
Note that if the original language was countable, so is the extension 
$\LL$. 

Consider now the ultraproduct $M$ of the $\LL$-structures 
$(M_n)_{n\in \mathbb N}$ with respect to some non-principal 
ultrafilter $\mathcal U$. Denote by $G$ and $X$ the corresponding 
interpretations in a  sufficiently saturated elementary extension 
$\UU$ of  $M$. For each $\LL$-formula $\varphi(x,y)$ and every 
tuple $b$ in $\mathbb U^{|y|}$ such that $\varphi(\mathbb U,b)$ 
is a subset of $\sbgp X$, define  
\[
\mu ( \varphi(x,b)) = \inf \left\{ r \in \mathbb Q^{\geq 0} \ | \ 
Q_{r,\varphi}(b) \text{ holds} \right\},
\]
where we assign $\infty$ if  $Q_{r,\varphi}(b)$ holds for no value 
$r$. This is easily seen to be a finitely additive definable measure on 
the ring $\Ring$ of definable subsets of $\sbgp X$ which is 
invariant under left and right translation. In particular, the pair 
$(G,X)$ is definably amenable.

We will throughout this paper consider two main examples:
	\begin{enumerate}[(a)]
		\item The set $X$ 	equals $G$ itself, which happens 
		whenever  the subset $X_n=G_n$ for $\mathcal U$-almost all 
		$n$ in $\N$.  The normalized counting measure $\mu$ defined 
		above  is a definable Keisler measure  
		\cite{jK87} on the pseudo-finite group $G$. Note that in this case the ring of sets $\Ring$ coincides with the Boolean algebra of all definable subsets of $G$.
		\item For $\mathcal U$-almost all $n$, the set $X_n$ has \emph{small tripling}: there is a constant $K>0$ such that $|X_nX_nX_n|\leq K|X_n|$.
  The non-commutative Pl\"unnecke-Ruzsa inequality \cite[Lemma 
		3.4]{tT08} yields  that $|X_n^{\odot 
		m}| \leq K^{O_m(1)}|X_n|$, so the measure $\mu(Y)$ is 
		finite for every definable subset $Y$ of $\sbgp X$, since $Y$ is 
		then  contained in $X^{\odot m}$ for some $m$ in $\N$. In 
		particular,  the corresponding $\sigma$-additive measure $\mu$ is  again $\sigma$-finite. 
\end{enumerate}
 Whilst  each subset $X_n$ in the example (b) must 
 be finite, we do 
 not impose that the groups $G_n$ are finite. If the set 
 $X_n$ has tripling at most $K$, the set $X^{\odot 1} = X_n\cup 
X_n\inv\cup \{\mathrm{id}_G\}$ has size at most $2|X_n|+1$ 
and tripling at most 
$(CK^C )^2 $ with respect to some explicit absolute
constant $C>0$. Thus, taking ultraproducts, both 
structures $(G,X)$ and 
$(G,X^{\odot 1})$ will have the same sets of positive measure (or 
density), though the values may differ.  Hence, we may 
assume that in a definably 
amenable pair $(G, X)$ the corresponding definable set $X$ is 
symmetric and  contains the neutral element of $G$. 
\end{example}

The above example can be adapted to consider countable amenable groups. 
\begin{example}\label{E:Foelner}
Recall that a countable group is \emph{amenable} if it is equipped 
with a sequence $(F_n)_{n\in\mathbb N}$ of  finite sets of increasing 
cardinalities (so  $\lim\limits_{n\to \infty} |F_n| = 	\infty$) such 
that all $g$ in $G$, \[
\lim\limits_{n\to\infty} \frac{|F_n \cap g \cdot F_n|}{|F_n|} = 1.
\] 

Such a sequence of finite sets is called a left F\o lner sequence. The archetypal example of an amenable group is $\Z$ with left F\o lner sequence $F_n=\{-n,\ldots,n\}$. 

By 
\cite[Corollary 5.3]{iN64}, if a group is amenable, then there is a distinguished left 
F\o lner sequence where each $F_n$ is symmetric. In particular,  the 
sequence $(F_n)_{n\in \N}$ is also a right F\o lner sequence: \[
\lim\limits_{n\to\infty} \frac{|F_n \cap F_n\cdot g|}{|F_n|} = 1 
\ \text{ for all $g$ in $G$}.\] 

Notice also that a subsequence of a F\o lner sequence is again F\o lner 
and so is the sequence $(F_n\times F_n)_{n\in \N}$ in the group 
$G\times G$. Given an amenable group $G$ with a distinguished F\o 
lner  sequence $(F_n)_{n\in\mathbb N}$ consisting of 
symmetric sets as well as a non-principal ultrafilter $\mathcal U$ on 
$\mathbb N$, the ultralimit
\[
\mu (Y) = \lim_{n\to \mathcal U} \frac{|Y\cap F_n|}{|F_n|},
\]
induces a finitely additive measure on 
the Boolean algebra of subsets of $G$ which is 
invariant under left and right translation.

Starting from a fixed countable language $\mathcal L$ expanding the 
language of groups, we can render the above measure definable, 
similarly as in Example \ref{E:amenable}. Hence, we can consider 
every countable amenable group $G$ as a definably amenable pair, 
setting $X=G$. 
\end{example}

\begin{example}\label{E:fsg}

Every stable group $G$ is fsg and thus  equipped with a unique left and right translation invariant Keisler measure which is generically stable (see \cite{HPS13} \& \cite[Example 8.34.]{pSbook}). 

Similarly,  a compact semialgebraic Lie group $G(\mathbb R)$, or more generally a definably compact group $G$ definable in an o-minimal expansion of a real closed field is again fsg. If the group is the $\mathbb R$-rational points of a  compact semialgebraic Lie group, this measure coincides with the normalised Haar measure. 

Hence, we can consider in these two previous cases (stable and o-minimal compact) the group $G$ as a definably amenable pair, setting $X=G$. 
\end{example}

If a group $G$ is definable, so is every finite cartesian product. Moreover, the construction in Example \ref{E:amenable} and \ref{E:Foelner} can also be carried 
out  for a finite cartesian product  to produce for every $n\geq 1$ in 
$\N$ a definably amenable pair $(G^n, X^n)$, where  $\langle X^{ 
n} \rangle = \sbgp X^{n}$, equipped with a definable 
$\sigma$-finite measure $\mu_n$. Thus, the 
following assumption 
is satisfied by our examples \ref{E:amenable}, \ref{E:Foelner} and \ref{E:fsg}. 
\begin{hyp}\label{H:sigmafte}
For every $n\ge 1$, the pair $(G^n ,X^n)$ is 
definably amenable for a definable $\sigma$-finite measure  
$\mu_n$ in a compatible fashion: the measure $\mu_{n+m}$ extends the corresponding product measure $\mu_n\times\mu_m$. 
\end{hyp} 

The definability condition in Definition \ref{D:meas_def} implies 
that the function 
\[ \begin{array}{rccl}
F^\varphi_{\mu_n,C}:	& S_{m}(C)  &\to &   \mathbb R, \\[1mm]	 & 	\tp(b/C) &\mapsto  & \mu_n(\varphi(x,b)) \end{array}\]
is well-defined  and continuous for every $\LL_C$-formula 
$\varphi(x,y)$ with 
$|x|=n$ and $|y|=m$ such that $\varphi(x,y)$ defines a subset of $\sbgp X^{n+m}$.  Therefore, for such $\LL_C$-formulae 
	$\varphi(x,y)$, consider the $\LL_C$-definable subset $Y=\{ y \in \sbgp X^m \ | \ \exists x \, \varphi(x,y)\}$ and the corresponding clopen subset $[Y]$ of $S_{m}(C)$. Thus,  we 
	can consider the following measure $\nu$ on $\sbgp 
	X^{n+m}$,
\[
\nu(\varphi(x,y)) = \int_{q\in [Y]} 
F^\varphi_{\mu_n,C}(q) \,d\mu_m= \int_{y\in Y} 
\mu_n(\varphi(x,y)) \,d\mu_m.
\] By an abuse of notation, we will write $\int_{\sbgp X^m} 
\mu_n(\varphi(x,y)) \,d\mu_m$ for $\int_{Y} 
\mu_n(\varphi(x,y)) \,d\mu_m$. 

For the pseudo-finite measures  
described in Example \ref{E:amenable}, the above integral equals the ultralimit 
\[ \lim_{k\to \mathcal U}\frac{1}{|X_k|^m}\sum_{y\in \langle X_k \rangle^m} \frac{|\varphi(x,y)|}{|X_k|^n} , \]
so $\nu$ equals $\mu_{n+m}$ and consequently Fubini-Tonelli 
holds, 
see (the proof of) \cite[Theorem 19]{BT14}. The same holds 
whenever the measure is given by densities with respect to a F\o lner 
sequence in an amenable group, as in Example \ref{E:Foelner}. For 
arbitrary definably amenable pairs, whilst the measure $\nu$ 
extends the product measure $\mu_n 
\times \mu_m$, it need not be {\it a priori} $\mu_{n+m}$ 
\cite[Remark 3.28]{sS17}. 
  Keisler \cite[Theorem 6.15]{jK87} exhibited a Fubini-Tonelli type 
theorem for general Keisler measures under certain conditions. These conditions hold for the unique generically stable translation invariant measure of an fsg group, see Example \ref{E:fsg}.   We will impose a 
further restriction on the definably amenable pairs we will 
consider, taking Examples \ref{E:amenable}, \ref{E:Foelner} and \ref{E:fsg}  as a guideline.

\begin{hyp}\label{H:Fubini}
For every definably amenable pair $(G, X)$ and its corresponding 
compatible system of definable measures $(\mu_n)_{n\in \mathbb N}$ on the 
Cartesian  powers of $\sbgp X$, the Fubini condition holds: 
Whenever a definable subset  of $\sbgp X^{ n+m}$ is given by an  
$\LL_C$-formula $\varphi(x,y)$ with $|x|=n$ and 
$|y|=m$, the following  equality holds:
\[
\mu_{n+m}(\varphi(x,y)) = \int_{\sbgp X^m} 
\mu_n(\varphi(x,y)) \,d\mu_m =  \int_{\sbgp X^n} 
\mu_m(\varphi(x,y)) \,d\mu_n.
\]  
(Note that the above integrals do not run over the locally definable 
sets $\sbgp X^m$ and $\sbgp X^n$, but rather over definable 
subsets, for $\varphi(x,y)$ is itself definable). 
\end{hyp}

Whilst this assumption is stated for definable sets, it 
extends to certain Borel sets, whenever the language $\LL_C$ is 
countable.   

\begin{remark}\label{R:Fubini_Borel}
	Assume that $\LL_C$ is countable and fix a natural number $k\ge 1$. Following \cite[Definition 2.6]{CGH23}, for every Borel subset $Z$ of $S_{n+m}(C)$ of types $q(x,y)$ with 
	$|x|=n$ and $|y|=m$, set \[ Z(x, b)=\left\{ p\in S_{n}(\UU)  \ | \ \tp(a, b/C) \text{ belongs to } Z \text{ for some  $a$ realizing } p\restr{C, b}\right\}.\]
 Note that $Z(x, b)$ only depends on $\tp(b/C)$ by \cite[Lemma 
 2.7]{CGH23}. If $Z$ is contained in the clopen set determined by the 
 $\LL_C$-definable set $(X^{\odot k})^{n+m}$, we define 
 analogously 
 as before a function
 \[ \begin{array}{rccl}
F^Z_{\mu_n,C}:	& S_{m}(C)  &\to &   \mathbb R, \\[1mm]	 & 	\tp(b/C) &\mapsto  & \mu_n(Z(x,b)) \end{array}.\]
This function is Borel, and thus measurable, by the definability of the 
measure as well as the monotone convergence theorem, for it agrees 
with $F^\varphi_{\mu_n,C}$ whenever $Z$ is the clopen 
$[\varphi]$. Furthermore, the following identity holds: 	\[
	\mu_{n+m}(Z(x,y)) = \int_{\sbgp X^m} 
	\mu_n(Z(x,y)) \,d\mu_m =  \int_{\sbgp X^n} 
	\mu_m(Z(x,y)) \,d\mu_n,
	\]
	by a straightforward application as 
	in \cite[Theorem 20]{BT14} of the monotone class theorem, using the fact that $\mu(X^{\odot 
	k})$ is finite. In particular, the above identity of integrals holds for 
	every Borel set of finite measure by regularity.
\end{remark}
\begin{remark}\label{R:Examples_Assumptions}
The examples listed in Examples \ref{E:amenable}, \ref{E:Foelner} and \ref{E:fsg}  satisfy both Assumptions \ref{H:sigmafte} and \ref{H:Fubini}.
\end{remark}

\medskip
{\bf  Henceforth, the language is countable and all definably amenable pairs satisfy Assumptions 
\ref{H:sigmafte} and \ref{H:Fubini}.}
\medskip

Adopting some terminology from additive combinatorics, we shall use 
the word \emph{density} for the value of the measure of a subset in 
$\Ring$ of a definably amenable pair  $(G,X)$.  A (partial) type is 
said to be \emph{weakly random} if it contains a 
definable subset in $\Ring$ of  positive density but no definable subset in $\Ring$ of 
density $0$. Note that every weakly random partial type 
$\Sigma(x)$ over a 
parameter set $A$ implies a definable set $X^{\odot k}$ in $\Ring$ for some $k$ in $\N$ and thus it can be completed to a weakly random complete type over 
any arbitrary set $B$ containing $A$, since the collection of formulae
\[ \Sigma(x)\cup\left\{X^{\odot k} \setminus Z \,|\, Z  \
\text{in $\Ring$ is $B$-definable of density $0$}\right\}\]
is finitely consistent. Thus, weakly random types exist 
(yet the partial type $x=x$ is not  weakly random whenever $G\neq 
\sbgp X$).  As usual, we say that an element $b$ of $G$ is 
weakly random over $A$ if  
$\tp(b/A)$ is.   

Weakly random elements satisfy a weak notion of transitivity.

\begin{lemma}\label{L:trans}
	Let $b$ be weakly random over a set of parameters $C$ and $a$ be 
	weakly random over $C, b$. The pair $(a,b)$ is weakly random 
	over $C$. 
\end{lemma}
\begin{proof}
	We need to show that every $C$-definable subset $Z$ of $\sbgp X^{n+m}$ 
	containing the pair $(a,b)$ has positive density with respect to the product 
	measure $\mu_{n+m}$, where $n=|a|$ and $m=|b|$. Since $a$ is weakly random over $C, b$, the fiber $Z_{b}$ of $Z$ 
	over $b$ has measure $\mu_n(Z_{b})= 2r$  for 
	some real number $0< r$. Hence $b$ belongs 
	a  $C$-definable subset $Y$ of \[ \left\{ y \in \UU^m \ |\   r \le \mu_n(Z_y) \le 
3r \right\},\]  by the definability of the measure. In particular, the measure  
	$\mu_m(Y)>0$. Thus, \[ \mu_{n+m}(Z) 
	=\int_{\sbgp X^m}
	\mu_n(Z_y) \,  d\mu_m \geq \int_{Y} \mu_n(Z_y) \,  d\mu _m 
	\geq 
	\mu_m(Y) r >0,\] as desired. 
\end{proof}
Note that the tuple $b$ above may not be weakly random over $C, 
a$. To remedy the failure of  symmetry in the notion of randomness, 
we will introduce  \emph{random} types, which will play a 
fundamental role in  Section \ref{S:pqr}. Though random types already appear in \cite[Subsection 2.23]{eH13} (see also \cite[Subsection 2.20]{eH12}), we will take the opportunity here to recall Hrushovski's definition of \emph{$\omega$-randomness}. All the ideas here until the end of this section are due to Hrushovski and we are merely writing down some of the details for the sake of the presentation. 

Fix some countable elementary substructure $M$ and some $Y$ in $\Ring$ definable over $M$ (so $Y\subseteq (X^{\odot k})$ for some $k$ in $\N$). As in Remark \ref{R:Cara}, we denote by $S_{Y^m}(M)$ the compact subset of the space of types over $M$ containing the $M$-definable subset $Y^m$.  

\begin{definition}\label{D:BooleanRandom} 
 Denote by $\BB$ the smallest Boolean algebra of subsets of $S_{Y^m}(M)$, as $m$ varies, containing all clopen subsets of $S_{Y^m}(M)$ and closed under the following operations: 
 \begin{itemize}
 \item The preimage of a set $W\subseteq S_{Y^m}(M)$ in $\BB$ 
 under the natural continuous  map $ S_{Y^n}(M)\to S_{Y^m}(M)$ 
 given by the restriction to a choice of $m$ coordinates belongs again 
 to $\BB$. 
 \item If $Z\subseteq S_{Y^{n+m}}(M)$ belongs to $\BB$, then so 
 does \[ (F^Z_{\mu_n,M})\inv (\{0\}) =\left\{ \tp(b/M) \in 
 S_{Y^m}(M) \ | \ \mu_n(Z(x, b))=0 \right\} ,\] with $Z(x, b)$ as 
 in Remark \ref{R:Fubini_Borel}.   
 \end{itemize}
 \end{definition}

Note that each element of $\BB$ is a Borel subset of the appropriate space of types by Remark \ref{R:Fubini_Borel}. Furthermore, it is countable since it can be  inductively built from the Boolean algebras of clopen subsets of the $S_{Y^m}(M)$'s  by adding in the next step all Borel sets of the form $(F^Z_{\mu_n,M})\inv (\{0\})$ and closing under  Boolean operations.  The collection $\BB$ contains new sets which are neither open nor closed. 
\begin{definition}\label{D:random}
Let $Y$ in $\Ring$ be definable over the countable 
	elementary substructure $M$. A $n$-tuple $a$ of elements in $Y$ 
	is \emph{random} over $M\cup B$, where $B$ is some countable 
	subset of parameters, if  $\mu_{n}(Z(x,b))>0$ for every finite 
	subtuple $b$ in $B$ and every Borel subset $Z$ in $\BB$ with 
	$\tp(a, b/M)$ in $Z$. 

For $B=\emptyset$, we simply say that the tuple is \emph{random} 
over $M$.
\end{definition}
\begin{remark}\label{R:random_wide} Since $\BB$ contains all 
clopen sets given by $M$-definable subsets, it is easy to see that a 
tuple random over $M\cup B$ is weakly random over $M\cup B$, 
which 
justifies our choice of terminology (instead of  using the term \emph{wide type} from \cite{eH12}). 

Randomness is preserved under the group law: If $a$ is an element of $\sbgp X$ random over $M\cup B$, then so are $a\inv$ and  $b\cdot a$ for every element $b$ in $B\cap \sbgp X$. 

Furthermore, note that randomness is a property of the 
type: If $a$ and  $a'$ have the same type over $M\cup B$, then $a$ is 
random over $M\cup B$ if  and only if $a'$ is. 
\end{remark}

\begin{remark}\label{R:posdens_random}
Since $\BB$ is countable, the $\sigma$-additivity of the measure 
yields that every measurable subset of $S_{Y^m}(M\cup B)$, with 
$B$ countable, 
 of positive density contains a random element over $M \cup B$. In 
 particular, every weakly random definable subset of $Y^m$ contains 
 random elements over $M, B$. 
\end{remark}

Randomness is a symmetric notion. 

\begin{lemma}\label{L:sym_random}\textup{(}\cite[Exercise 
2.25]{eH13}\textup{)} 
Let $Y$ in $\Ring$ be definable over the countable 
	elementary substructure	$M$. A finite tuple $(a,b)$ of elements in $Y$ is random over $M$ if 
and only if  $b$ is 
random 
over $M$ and $a$ is random over $M,b$. 
\end{lemma}
\begin{proof}
 Assume first that $(a,b)$ is random over 
$M$. Clearly so is $b$ by Fubini and Remark \ref{R:Fubini_Borel}. 
Thus we 
need only prove that $a$ is random over $M, b$. Suppose for a  
contradiction that $\mu_{|a|}(Z(x, b))=0$ for  some $Z \subseteq 
S_{Y^{|a|+|b|}}(M)$ 
of $\BB$  
containing $\tp(a,b)$.  The type of the pair $(a,b)$ belongs to 
\begin{multline*}    
\widetilde Z= Z\cap \pi\inv\left((F^{Z}_{\mu_{|a|},M})\inv(\{0\})\right) = \\ Z\cap \left\{\tp(c,d/M) \in S_{Y^{|a|+|b|}}(M) \ | \ 
\mu_{|a|}(Z(x, d)) =0\right\},\end{multline*}  
where $\pi=\pi_{|a|+|b|,|b|}$ is the corresponding restriction map. 
Now, the set $\widetilde Z$
belongs to $\BB$ and contains $(a,b)$, so it cannot have density 
$0$. However, Remark \ref{R:Fubini_Borel} yields \[ 0< \mu_{|a|+|b|}(\widetilde Z) 
= \int_{Y^{|b|}} \mu_{|a|}(\widetilde{Z}(x, d)) \, d\mu_{|b|} \le \int_{Y^{|b|}} \mu_{|a|}(Z(x, d)) \, d\mu_{|b|} =0 
,\] 
which gives the desired contradiction. 

Assume now that $b$ is random over $M$ and $a$ is random over 
$M, b$. Suppose for a contradiction that $\tp(a,b/M)$  lies in some Borel $Z(x,y)$ of $\BB$ with $\mu_{|a|+|b|}(Z)=0$. By Remark \ref{R:Fubini_Borel}, 
\[ 
0=\mu_{|a|+|b|}(Z)= \int_{Y^{|b|}} \mu_{|a|}(Z(x, d)) \, d\mu_{|b|},  
\] so $\mu_{|a|}(Z(x, d))=0$ for $\mu_{|b|}$-almost all 
types $\tp(d/M)$ in 
$S_{Y^{|b|}}(M)$.  Hence, the set  
$(F^{Z}_{\mu_{|a|},M})\inv(\{0\})$ has measure 
$\mu_{|b|}(Y^{|b|})$. Since $a$ is random over $M, b$, we have that 
$\mu_{|a|}(Z(x, b))>0$, so $\tp(b/M)$ belongs to the 
complement of $(F^{Z}_{\mu_{|a|},M})\inv(\{0\})$, which belongs 
to $\BB$ and has $\mu_{|b|}$-measure $0$. We conclude that the 
element $b$ 
is not random over $M$, which gives the desired contradiction. 
\end{proof}

Symmetry of randomness will allow us in Sections \ref{S:pqr} and \ref{S:ultra} to transfer ideas arisen from the study of definable groups in simple theories to the pseudo-finite context as well as to definably compact groups definable in o-minimal expansions of real closed fields. Whilst weakly randomness is not symmetric, a weak form of symmetry holds (as  pointed out by the anonymous referee, to whom we would like to express our sincere gratitude again). 

\begin{lemma}[The referee's lemma]\label{L:ref}
Let $Y$ in $\Ring$ be a subset of positive density definable over the countable elementary substructure $M$. Given  two finite tuples $a$ and $b$ of elements in $Y$ with $a$ weakly random over $M$ and $b$ random over $M,a$, then $a$ is weakly random over $M,b$.
\end{lemma}
\begin{proof}
Assuming otherwise, there is an $M$-definable set $Z$ containing $(a,b)$ such that the fiber $Z_b$ has $\mu_{|a|}$-measure $0$. Definability of the measure \ref{D:meas_def}  yields that the set 
\[
W=(F_{\mu_{|a|},M}^Z)^{-1}(\{0\})= \left\{ \tp(d/M) \in S_{Y^{|b|}}(M) \ | \ \mu_{|a|}(Z_d)=0\right\}
\] 
is closed and thus it can be written as a countable intersection $W=\bigcap_{m\in \N} W_m$ of $M$-definable sets with $W_{m+1}\subseteq W_m$. Now, the closed set $[Z(x, y)] \cap W(y)$ belongs to $\BB$ and contains $\tp(a, b/M)$, so $\mu_{|b|}([Z(a, y)]\cap W(y))>0$, since $b$ is random over $M,a$.

\begin{claim*}
There exists some $M$-definable subset $V$ containing $a$ such that \[ \mu_{|b|}([Z(a', y)]\cap W(y))>0\]  for all $a'$ in $V$. 
\end{claim*}
Note that $V$ has positive density, for $\tp(a/M)$ is weakly random.  
\begin{claimproof*}
Assume for a contradiction that this is not the case. Since both the language and $M$ are countable, we may list all $M$-definable subsets containing $a$ as $\{V_n\}_{n\in \N}$ with $V_{n+1}\subseteq V_n$. Therefore, for every $n$ in $\N$ there is some $a_n$ in $V_n$ with $\mu_{|b|}([Z(a_n, y)]\cap W(y))< \frac{1}{n+1}$. As $W$ is a countable intersection of the $W_m$'s, there is some $m_n$ in $\N$ such that 
\[
\mu_{|b|}\left(Z(a_n, y) \cap W_{m_n}(y)\right) < \frac{1}{n+1}.
\] Notice that we may construct the sequence such that $m_{n+1} > m_n$. 
Set \[ \theta_{<} (Z, W_{m_n}) =\left\{ x\in Y^{|a|} \ | \ \mu_{|b|}(Z(x, y) \cap W_{m_n}(y)) <\frac{1}{n+1}\right\}\] and  define $\theta_{\le}(Z,W_{m_n})$ analogously.  By definability of the measure, there is some $M$-definable subset $\theta(Z, W_{m_n})$ such that 
\[ \theta_<(Z, W_{m_n}) \subseteq  \theta(Z, W_{m_n}) \subseteq \theta_\le(Z, W_{m_n}).\]
In particular, we have that $\theta(Z,W_{m_{n+1}})\subseteq \theta(Z,W_{m_n})$ for $m_{n+1}> m_n$.
Now, the collection of $\LL_M$-formulae 
$\{V_n(x) \land \theta(Z, W_{m_n})(x) \}_{n\in \N}$ cannot be consistent, for it would yield the existence of a tuple $a'$ realizing $\tp(a/M)$ with 
\[ 
\mu_{|b|}\left([Z(a', y)] \cap W(y)\right) \le \mu_{|b|}\left(Z(a', y) \cap W_{m_n}(y)\right) \le \frac{1}{n+1}\] for every $n$ in $\N$, so $ \mu_{|b|}\left([Z(a', y)] \cap W(y)\right) =0 < \mu_{|b|}([Z(a, y)]\cap W(y))$, which is a contradiction. By compactness, there exists some $\ell$ in $\N$ such that no realization of $V_\ell$ satisfies some $\theta(Z, W_{m_j})$ with $j\le \ell$. However, the element $a_{\ell}$ belongs to $V_\ell\cap \theta_<(Z, W_{m_\ell})$, so $a_{\ell}$ lies in every $\theta(Z, W_{m_j})$ with $j\le \ell$, which gives the desired contradiction. 
\end{claimproof*}

Consider now the closed set $W'=[V(x)]\cap [Z(x,y)]\cap W$. The Fubini condition \ref{R:Fubini_Borel} yields that 
\begin{multline*} 
0\stackrel{\text{ Claim }}{<}  \int_{\tp(c/M)\in [V]}  
	\mu_{|b|}([Z(c,y)]\cap W)  \,d\mu_{|b|} = \\ = \mu(W')=   \int_{\tp(d/M) \in W} 
	\mu_{|a|}(V(x)\cap Z(x,d)) \,d\mu_{|a|} \le \\  \le \int_{\tp(d/M) \in W} 
	\mu_{|a|}(Z(x,d)) \,d\mu_{|a|}= 0.
 \end{multline*} 
We deduce from the above contradiction that $a$  lies in  no  definable set $Z_b$ over $M, b$ of density $0$, so $a$ is weakly random over $M, b$, as desired.
\end{proof}

\section{Forking and measures}\label{S:stable} 
As in Section \ref{S:Fubini}, we work inside a sufficiently saturated 
structure and a definably amenable pair $(G, X)$ in a fixed 
countable language $\LL$ satisfying Assumptions 
\ref{H:sigmafte} and \ref{H:Fubini}, though the classical 
notions of forking 
and stability do not 
require the presence of a group nor of a measure. 

Recall that a definable set $\varphi(x,a)$ \emph{divides} over a 
subset $C$  of parameters if  there exists an indiscernible sequence 
$(a_i)_{i\in \N}$ over $C$ with $a_0=a$ such that the intersection 
$\bigcap_i \varphi(x,a_i)$ is empty. Archetypal examples of 
dividing formulae are of the form $x=a$ for some element $a$ not 
algebraic over $C$. Since dividing formulae need not be closed 
under finite disjunctions, witnessed for example by a circular order, 
we say 
that a fomula $\psi(x)$ \emph{forks} over $C$ if it belongs to the 
ideal generated by the formulae dividing over $C$, that is, if $\psi$ 
implies a finite disjunction of formulae, each dividing over $C$. A 
type \emph{divides}, resp. \emph{forks} over $C$, if it contains an 
instance which does. 

\begin{remark}\label{R:posdens_nf}
Since the measure is invariant under automorphisms and 
$\sigma$-finite, no definable subset of  $\sbgp X$ of 
positive density divides, and thus no weakly random type forks 
over the empty-set, see \cite[Lemma 2.9 \& Example 2.12]{eH12}.
\end{remark}

Non-forking need not define a tame notion of independence, for 
example it need not be symmetric, yet it behaves extremely well 
with respect to certain invariant relations, called stable. 

\begin{definition}\label{D:stable}
An $A$-invariant relation $R(x,y)$ is \emph{stable} if there is no 
$A$-indiscernible sequence $(a_i, b_i)_{i \in \N}$ such that \[ R(a_i, 
b_j)  \text{ 
holds if and only if }  i<j.\]
\end{definition}
A straight-forward Ramsey argument yields that the collection of 
invariant stable relations is closed under Boolean combinations.  
Furthermore, an $A$-invariant  relation  
is stable if there is no $A$-indiscernible sequence as in the 
definition of  length some fixed infinite ordinal. 

The following remark will be very useful in the following sections.

\begin{remark}\label{R:stable_sym}\textup{(}\cite[Lemma 
2.3]{eH12}\textup{)} 
Suppose that the type 
$\tp(a/M, b)$ does not fork over  the  elementary substructure $M$ 
and that the $M$-invariant relation $R(x,y)$ is stable. 
Then the 
following are equivalent:
\begin{enumerate}[(a)]
\item The relation $R(a,b)$ holds. 
\item The relation $R(a', b)$ holds, whenever $a'\equiv_M a$ and 
$\tp(a'/Mb)$ does not fork. 
\item The relation $R(a', b)$ holds, whenever $a'\equiv_M a$ and 
$\tp(b/Ma')$ does not fork. 
\item The relation $R(a', b')$ holds, whenever $a'\equiv_M a$ and $b'\equiv_M a$ such that $\tp(a'/M, b')$ or $\tp(b'/M, a')$ does not fork. 
\end{enumerate}
\end{remark}

A clever use of  the Krein-Milman theorem on the locally 
compact Hausdorff topological real vector space of all $\sigma$-additive probability measures allowed Hrushovski to  
prove the  following striking result (the case $\alpha=0$ is an easy consequence of the inclusion-exclusion principle): 
\begin{fact}\label{F:stable_measure0}\textup{(}\cite[Lemma  2.10 
\& Proposition
 2.25]{eH12}\textup{)} 
Given a real number $\alpha$ and $\LL_M$-formulae $\varphi(x, 
z)$ and $\psi(y, z)$ with parameters over an elementary substructure $M$, the $M$-invariant relation on the definably 
amenable pair $(G,X)$ \[ R^\alpha_{\varphi, \psi}(a,b) 
\Leftrightarrow \mu_{|z|}\big(\varphi(a,z) \land 
\psi(b,z)\big)=\alpha\] is stable. 
In particular, for any partial types $\Phi(x,z)$ and $\Psi(y,z)$ over $M$, the relation 
\[
Q_{\Phi, \Psi}(a,b) \Leftrightarrow  \ \Phi(a,z) \land \Psi(b,z) \text{ is weakly random}
\]  is stable. 
\end{fact}
Strictly speaking, Hrushovski's result in its 
original version is stated for arbitrary Keisler measures (in any theory). 
To deduce the statement above it suffices to normalize the measure 
$\mu_{|z|}$ by $\mu_{|z|}((X^{|z|})^{\odot k})$ for some natural 
number $k$  such that $(X^{|z|})^{\odot k}$ contains the 
corresponding instances of $\varphi(x,z)$ and $\psi(y,z)$.

We will finish this section with a summarized version of Hrushovski's 
stabilizer theorem tailored to the context of definably amenable 
pairs. Before stating it, we first need to introduce some notation.

\begin{definition}\label{D:G00}
Let $X$ be a definable subset of a definable group $G$ and let $M$ be 
an elementary substructure. We denote by $\Conn M$  the 
intersection of all subgroups of $\sbgp X$ type-definable over $M$  
and of 
bounded index. 
\end{definition}
If a subgroup of bounded index  type-definable over $M$ 
exists, the 
subgroup $\Conn M$ is again type-definable 
over $M$ and has bounded index, see \cite[Lemmata 3.2 \& 
3.3]{eH12}. Furthermore, it is also normal in  $\sbgp X$ \cite[Lemma 3.4]{eH12}.

\begin{fact}\textup{(}\cite[Theorem 3.5]{eH12} \& 
\cite[Theorem
	2.12]{MOS18}\textup{)}\label{F:Hr}
	Let $(G,X)$ be a definably amenable pair and let $M$ be an elementary substructure. The subgroup $\Conn M$ exists and equals
	\[
	\sbgp X_M^{00} = (p\cdot p\inv)^2,
	\] for any weakly random type $p$ over $M$,	where we identify a type with 
	its realizations in the ambient structure $\UU$. 
	Furthermore, the set $pp\inv p$ is a coset of $\Conn M$. For 
	every  
	element $a$ in $\Conn M$  weakly random  over $M$, the partial 
	type  $p \cap a\cdot p$ is weakly random.  In particular, every weakly random element in $\Conn M$ over $M$ lies in $p\cdot p\inv$. 
\end{fact}

If the definably amenable pair we consider happens to be as in the 
first case of  Example \ref{E:amenable} or a stable group as in Example \ref{E:fsg}, our notation coincides with the classical notation $\GO M$.

\section{On $3$-amalgamation and solutions of $xy=z$}\label{S:pqr}

As in Section \ref{S:Fubini}, we fix a definably amenable pair $(G, 
X)$ satisfying Assumption \ref{H:sigmafte} and \ref{H:Fubini}. All throughout this section, we work over some fixed elementary  substructure $M$. We denote by $\supp_M(\mu)$ the 
\emph{support} of $\mu$, that is, the collection of all weakly random types over $M$ 
contained in $\sbgp X$.  

Note that each coset of the subgroup $\Conn M$ of Definition \ref{D:G00} is type-definable over $M$ and 
hence $M$-invariant, though it 
need not have a representative in $M$. Thus, every type $p$ over 
$M$ 
contained in $\sbgp X$ must determine a coset of $\Conn M$. We 
denote by $\Cos p$ the 
coset of $\Conn M$ of $\sbgp X$ containing some (and hence every)
realization of $p$. The following result resonates with \cite[Corollary 1]{TW19} and \cite[Theorem 1.3]{CPT20} beyond the definable context. 
 
\begin{prop}\label{P:stable_invariant}
Consider an $M$-invariant subset $S$ of $\sbgp X$ such that the relation $u\cdot v \in S$ is stable, as in Definition \ref{D:stable}. The set $S$ must be, up to $M$-definable sets of measure $0$, a union of cosets of $\Conn M$, that is, if an element $g$ in $\sbgp X$ belongs to $S$ with $q=\tp(g/M)$ in $\supp_M(\mu)$, then every  element $h$ in $\Cos{q}$ weakly random over $M$ belongs to $S$ as well. 
\end{prop}
Our proof is mostly an adaptation of \cite[Proposition 2.2]{PSW98}. Whilst the authors used the independence theorem from simple theories, we will use the stability of the $M$-invariant relation $S$  instead.
\begin{proof}
Assume that the element $g$ as above belongs to the stable $M$-invariant relation $S$. Let $h$ be in $\Cos{\tp(g/M)}$ weakly random over $M$ and choose a 
realization $b$ of $\tp(h/M)$ weakly random over $M, g$.  Now, the elements $g$ and $b$ both lie in the same coset of $\Conn M$, so  
the difference $b\cdot g\inv$ lies in $\Conn M$ and is weakly 
random over $M, g$. Since 
weakly random types do not fork, the type $\tp(b\cdot g\inv /M,g)$ 
does not fork 
over $M$. 

Fact \ref{F:Hr} yields that the partial type 
$\tp(g/M)\cap (b \cdot g\inv)\cdot \tp(g/M)$ is  weakly random. Choose therefore some element $g_1$ realizing 
$\tp(g/M)$ 
weakly random over $M, g, b$ such that $b\cdot g\inv \cdot 
g_1\equiv_M g$. By invariance of $S$, 
we have that $b\cdot g\inv\cdot g_1$ belongs to $S$ as well.  

Summarizing, the $M$-invariant relation $\bar S=\left\{ (u,v) \in \sbgp X\times \sbgp X \ | \ u\cdot v\in S\right\}
$ holds for 
the pair $(b\cdot g\inv, 
g_1)$ with  
$\tp(g_1/M, b\cdot g\inv)$ weakly random and hence non-forking 
over $M$. Since the above 
relation is stable, for  any pair $(w,z)$ such that \[ w\equiv_M 
b\cdot g\inv\ , \   
z\equiv_M  
g_1 \ \text{ and }  \ \tp(w/M,z) \ \text{ non-forking over } M,\]   the relation  
$\bar S$ must also hold. Setting now $w=b\cdot g\inv$ and $z=g$, 
we 
deduce that $ b = b\cdot g\inv \cdot g$ belongs to  $S$. As the 
element $h$ realizes $\tp(b/M)$, we conclude by $M$-invariance  
that $h$ belongs to $S$, as desired.
\end{proof}
Given now two $M$-definable subsets $A$ and $B$,  the relation \[ R^\alpha_{A,B}(u, v) \ \Leftrightarrow \ ``\mu( u A \cap 
v B ) =\alpha \text{''} \] 
is stable by Fact \ref{F:stable_measure0}. So, setting $S=\{g\in \sbgp X \ | \ \mu( A \cap 
g B ) =\alpha\}$, Proposition 
\ref{P:stable_invariant} yields immediately the following result, which we personally think it resonates with Croot-Sisask's almost-periodicity \cite[Corollary 1.2]{CS10}.

\begin{cor}\label{C:measure_constant_coset}
	Given two $M$-definable subsets $A$ and $B$, 
	the values $\mu(A\cap gB)$ and $\mu(A\cap hB)$ agree for any 
	two weakly random 
	elements $g$ and $h$ over $M$ within the same coset of $\Conn 
	M$. \qed
\end{cor}

Given now two types $p_1$ and $p_2$ over $M$ and an element $g$ of $\sbgp X$ such that the partial type $p_1\cdot g \cap p_2$ is consistent, it follows that the type $\tp(g/M)$ determines the coset $\Cos{p_1}\inv\cdot \Cos{p_2}$, so $\Cos{p_1} \cdot \Cos{\tp(g/M)} = \Cos{p_2}$. The following result can be seen as a sort of converse. Notice that \[S=\{g \in \sbgp X \ | \ p_1\cdot g \cap p_2 \text{ is weakly random over }M \}\]
is $M$-invariant and $u\cdot v\in S$ is stable, by Fact
\ref{F:stable_measure0}. 

\begin{cor}\label{C:weakly_random_everywhere_cos}
Let $p, q$ and $r$ be three coset-compatible types in $\supp_M(\mu)$, 
that is,\[ \Cos p \cdot \Cos q = \Cos r.\] If $p\cdot g\cap r$ is weakly 
random for some element $g$ in $\sbgp X$ with $\tp(g/M)$ in 
$\supp_M(\mu)$, then so is $p\cdot h \cap r$ for every weakly 
random element $h$ whose type over $M$ is concentrated in $\Cos{q}$. \qed
\end{cor}

The above result was first observed for principal generic types in a simple 
theory in \cite[Proposition 2.2]{PSW98} and later generalized to 
non-principal types in \cite[Lemma 2.3]{MPP04}. For  weakly random types with respect to a  pseudo-finite Keisler  measure, a preliminary version was obtained by the second author \cite[Proposition 3.2]{dP20} for 
ultra-quasirandom groups.

For the rest of this section, we will assume that $M$ is countable. Fix some $k$ in $\N$ and consider $Y=(X^{\odot k})$. The value $k$ should be chosen large enough to ensure that all the products and inverses of elements in the subsequent still belong $Y$. By an abuse of language, we will use the word \emph{random} to mean a random type with respect to the corresponding class $\BB$ as in Definitions \ref{D:BooleanRandom} and \ref{D:random}.   

\begin{remark}\label{R:random_dense}
It follows immediately from Remark \ref{R:posdens_random} that the Borel set of random types over $M$ is dense in the compact Hausdorff space of weakly randoms concentrated on $Y$, that is, the space $[Y]\cap \supp_M(\mu)$, where $[Y]$ is the clopen set given by the $M$-definable set $Y$. We denote by $\mathrm{R}(\mathcal B^Y_M)$ the collection of random types over $M$ concentrated on $Y$.
\end{remark}

\begin{lemma}\label{L:inter_posdens}
Given $M$-definable subsets $A$ and $B$ of $Y$ of 
positive 
density,  there exists some 
random element $g$ over $M$  with $\mu(Ag\cap  B)>0$.  
\end{lemma}
\begin{proof}
By Remark \ref{R:posdens_random}, let $c$ be random in $B$ 
over $M$ and choose now $g\inv$ in $c\inv A$ random 
over 
$M, c$. The element $g$ is also random over $M,c$. By symmetry of 
randomness, the pair $(c,g)$ is random over $M$, so $c$ is random over $M, g$.  
Clearly the element $c$ lies in $Ag\cap B$, so the set $Ag\cap B$ has positive 
density, as desired. 
\end{proof}
\begin{remark}\label{R:inter_posdens}
Notice that the above results yields the existence of an element $h$ random 
over $M$ such that $hA\cap  B$, and thus $A\cap h\inv B$, has positive density: 
Indeed, apply the 
statement to the definable subsets $B\inv$ and $A\inv$. 
\end{remark}

For any two fixed types $p$ and $r$ in $\supp_M(\mu)$, the statement 
\[  ``p\cdot y \cap r \text{ is weakly random } \& \text{ $y$ is 
weakly random''},\] as a property of $y$ is finitely consistent: Indeed, 
given finitely many $M$-definable subsets $A_1, \ldots, A_n$ in $p$ 
and $B_1, \ldots, B_n$ in 
$r$, the $M$-definable subsets $A=\bigcap_{1\le i\le n}A_i$ and 
$B=\bigcap_{1\le i\le n}B_i$ lie in $p$ and $r$ respectively, so  
they both have positive density. By Lemma \ref{L:inter_posdens}, there exists a random element $g$ in $\sbgp X$ over $M$ with $A_ig\cap B_j$ of 
positive density for all $1\le i, j\le n$. 

However, the condition ``$p\cdot y \cap r$  is weakly random'' is a 
$G_\delta$-condition on $y$, namely 
\[
\bigcap_{\substack{ A\in p\\ B\in r}}\left\{ y \in A\inv B \ | \ \mu(A\cdot y\cap B)>0\right\}.
\] 
Thus, we cannot use compactness to 
deduce from the above argument that we fulfill the conditions of 
Corollary \ref{C:weakly_random_everywhere_cos} for all weakly 
random types $p, q$ and $r$. We are grateful to Angus Matthews for 
pointing out a mistake in a previous version of this paper. 

To circumvent the aforementioned issue, we shall use the so-called 
\emph{disintegration theorem} which will allow us to fulfill the conditions of 
Corollary \ref{C:weakly_random_everywhere_cos} for almost all 
pairs of types $p$ and $r$. Whilst there are plenty of excellent 
references on this subject worth being named, we just refer to 
\cite{vB07, dS12}. 

\begin{remark}\label{R:logic_top}
Given $n$ in $\N$ consider a set $\Omega$ and a surjective 
map $ F:S_{Y^n}(M)\to \Omega$ such that the set $\{ (p, q) \in 
S_{Y^n}(M)\times S_{Y^n}(M) \ | \ F(p)=F(q)\}$ is closed. For 
example, consider a type-definable equivalence relation $E(x,y)$ on  
$Y^n\times Y^n$ with parameters over $M$ and set $p\sim q$ if 
and only if \[ p(x)\cup q(y)\cup E(x,y) \text{ is consistent}.\] The 
relation $\sim$  is a closed equivalence relation on  $S_{Y^n}(M)$, 
so set $\Omega$ to be the collection of $\sim$-equivalence classes and $F$ 
the natural projection map. 

We can now equip $\Omega$ with the final topology with respect to $F$, so a subset $C$ of $\Omega$ is closed if and only if $F\inv(C)$ is closed in the topological space $S_{Y^n}(M)$. It is immediate to see that $\Omega$ with this topology becomes a compact Hausdorff separable space. 
Furthermore,  we can define a measure on $\Omega$, the  
\emph{push-forward measure} $F_*\mu$,  given by 
$F_*\mu(B)=\mu(F\inv(B))$ for every Borel subset $B$ of 
$\Omega$. 
\end{remark}
\begin{fact}\textup{(}Disintegration theorem\textup{)}\label{F:disint}
Consider the normalized measure $\mu_{Y^n}$ on the space of types 
$S_{Y^n}(M)$, so it becomes a probability space. Given 
$F:S_{Y^n}(M)\to \Omega$ as in Remark \ref{R:logic_top}, there 
exists a \emph{disintegration} of $\mu_{Y^n}$ by a (uniquely 
determined) 
family of Radon conditional probability measures on $S_{Y^n}(M)$ 
with respect to the continuous function $F:S_{Y^n}(M)\to \Omega$, 
{\it i.e.} there exists a mapping \[ (Z, t)\mapsto \nu(Z, 
t)=\mu_t(Z),\] 
where 
$Z$ is a Borel set of $S_{Y^n}(M)$ and $t$ is an element of 
$\Omega$, with the following properties:
\begin{enumerate}[(a)]
	\item\label{I:reg} for  all $t$ in  $\Omega$, the measure 
	$\mu_{t}$ is a Borel inner regular probability  measure on 
	$S_{Y^n}(M)$;
	\item\label{I:meas} for every measurable subset $Z$ of $S_{Y^n}(M)$, the function $t\mapsto \mu_{t}(Z)$ is measurable with respect to the measure $F_*\mu_{Y^n}$;
 \item\label{I:concent} each measure $\mu_t$ is concentrated on the 
 fiber $F\inv (t)$, that is, the measure $\mu_t( 
 S_{Y^n}(M)\setminus F\inv(t))=0$, so $\mu_t(Z)= \mu_t(Z\cap 
 F\inv(t))$ for every Borel subset $Z$ of $S_{Y^n}(M)$; 
	\item\label{I:fonct} for every measurable function 
	$f:S_{Y^n}(M)\to \mathbb R$, we have that  
 \[
\int_{S_{Y^n}(M)} f \ d\mu_{Y^n} = \int_{t\in\Omega} \int_{F\inv(t)} f \ d\mu_t \, dF_*\mu_{Y^n}.  
\] 
In particular, setting $f$ the characteristic function $\mathbbm{1}_Z$ of the measurable 
subset 
$Z$ of $S_{Y\times Y}(M)$, we have that 
 \[
\mu_{Y^n}(Z) = \int_{t\in\Omega} \mu_t(Z) \, dF_*\mu_{Y^n}. 
 \]
\end{enumerate}
\end{fact}

\begin{lemma}\label{L:exist_random}
Consider the natural restriction map \[ 
\begin{array}{rccl}

 \pi:	& S_{Y^2}(M)  &\to &   S_Y(M) \times S_Y(M)  \\[1mm]	 & 	q(y_1,y_2) &\mapsto  & (q\restr {y_1}(y_1), q\restr {y_2}(y_2))  \end{array}.
 \] 
Every pair of types $(p, r)$ of $S_Y(M) \times S_Y(M)$ outside of  a $\pi_*\mu_{Y^2}$-measure $0$ set can be completed to a random type of $S_{Y^2}(M)$.
\end{lemma}
\begin{proof}
Let $\mathrm{R}(\mathcal B_M^{Y^2})$ be the Borel  set of random 
types on $S_{Y^2}(M)$. It follows from Remark 
\ref{R:posdens_random} that $\mu_{Y^2}(\mathrm{R}(\mathcal 
B_M^{Y^2}))=1$. Apply now  the 
disintegration theorem (Fact \ref{F:disint}) with $\Omega=S_Y(M) 
\times S_Y(M)$ and $F=\pi$, and deduce from 
\[
1=\mu_{Y^2} ( \mathrm{R}(\mathcal B_M^{Y^2}) ) = \int_{(p,r)\in 
S_Y(M)\times S_Y(M)} \mu_{(p,r)} ( \mathrm{R}(\mathcal 
B_M^{Y^2})) \, d\pi_*\mu_{Y^2} 
\] that $\mu_{(p,r)} ( \mathrm{R}(\mathcal B_M^{Y^2})) =1$ 
for $\pi_*\mu_{Y^2}$-almost all pairs $(p,r)$,  since each function 
$\mu_{(p,r)}$ takes values in the interval $[0,1]$. In  particular, the 
set $\pi\inv(p,r)\cap \mathrm{R}(\mathcal B_M^{Y^2})$ 
is non-empty by 
Fact \ref{F:disint} (\ref{I:concent}). Every such completion yields a 
random pair $(a, b)$ 
over $M$, with $a$ realizing $p$ and $b$ realizing $r$, as desired. 
\end{proof}
\begin{theorem}\label{T:pqr}
For every pair of types $(p, r)$ of $S_Y(M) \times S_Y(M)$ outside of  a $\pi_*\mu_{Y^2}$-measure $0$ set and every weakly random type $q=\tp(b/M)$ concentrated on $Y$ with $\Cos p\cdot  \Cos q= \Cos r$, 
 there is a realization $a$ of $p$  weakly random over $M,b$ such that $a \cdot b$ realizes $r$.
\end{theorem}
\begin{proof}
By Lemma \ref{L:exist_random}, for every 
pair 
$(p, r)$ of $S_Y(M) \times S_Y(M)$ outside of  a 
$\pi_*\mu_{Y^2}$-measure $0$ set there exists  a random pair $(c, d)$ 
over $M$, with $c$ realizing $p$ and $d$ realizing $r$. By Remark  
\ref{R:random_wide} and Lemma 
\ref{L:sym_random}, the pair  $(c\inv\cdot 
d, d)$ is random over $M$, so the partial type $p\cdot (c\inv\cdot 
d)\cap r$ admits a random realization, and thus it is  weakly 
random. The element $c\inv\cdot d$ is (weakly) random over $M$ 
and belongs to $\Cos q$, since $p$, $q$ and $r$ are coset-compatible. 
We can thus apply Corollary 
\ref{C:weakly_random_everywhere_cos} to deduce that $p\cdot b 
\cap r$ is weakly random.  Choose  some realization $f$ of this 
partial 
type weakly random over $M,b$ and notice that the element 
$a=f\cdot b\inv$ is weakly random over $M, b$ and realizes $p$. By 
construction, the product $a\cdot b=f$ realizes $r$, as desired.
\end{proof}

Whilst Theorem \ref{T:pqr} holds for almost all types $(p,r)$, the corresponding $\pi_*\mu_{Y^2}$-measure $0$ set could possibly contain all diagonal pairs $(p,p)$, with $p$ in $\supp_M(\mu)$. We will conclude this section with an elementary observation, the consequences of which will be explored in detail in Section \ref{S:ultra}. 

\begin{remark}
Fix a countable elementary substructure $M$. If there exist a 
random pair $(a, b)$ over $M$ with $a\equiv_M b$, then there 
exists a random type concentrated in $\Conn M$.  Indeed, the 
element $b\inv \cdot a$ is random over $M$ by Remark 
\ref{R:random_wide} and Lemma \ref{L:sym_random}. Clearly, the 
element $g=b\inv\cdot a$ lies in $\Conn M$, as desired. 
\end{remark}

\begin{question}
Is there a a random pair $(a, b)$ over $M$ with $a\equiv_M b$? More generally, 
is there a random type concentrated in $\Conn M$?
\end{question}

\section*{A digression: Roth's theorem on arithmetic progressions}

We will now show how Corollary \ref{C:measure_constant_coset} yields 
 solutions to the equation $x\cdot z = y^2$ in subsets of positive density 
 for every definably amenable pair  such that the squaring function $x\mapsto x^2$ preserves randomness.
 \begin{definition}\label{D:preserve_random}
The function $f:X \to G$ in the definably amenable pair $(G,X)$ \emph{preserves randomness} if for every element $a$ in $X$ and every subset $C$ of parameters,  we have that $a$ is (weakly) random over  C if and only if $f(a)$ is (weakly) random over $C$ (so $f(a)$ must lie in $\sbgp X$).
\end{definition}

 \begin{remark}\label{R:preserv_weak_random}
    The examples \ref{E:amenable},  \ref{E:Foelner} and \ref{E:fsg} 
    always have the property that the square function preserves 
    randomness if the map $f:X\to G$ defined by $f(x)=x^2$ has 
    finite fibers. This is always the case whenever $X$ has distinct squares as in \cite[Theorem 1.5]{tS17} or if $G$ is abelian and 
    there are only finitely many involutions in $\sbgp X$.   
\end{remark}

\begin{theorem}\label{T:Roth_defamen}
    Consider a definably amenable pair $(G,X)$ such that the square function preserves randomness. If the definable subset $A$ of $X$ has positive density, then the set 
     \[ 
  \left\{ (x_1, x_2) \in A\times A \ | \ x_1\cdot x_2  \in A^2 \right\}
     \] has positive $\mu_2$-density, where $A^2=\{x^2\}_{x\in A}$.
\end{theorem}
Assume $A$ is definable over the countable elementary substructure $M$. Every pair $(a, c)$ in the above set random over $M$ gives raise to a 
\emph{generalised} $3$-AP in 
$A$. Indeed, the product $a\cdot c$ belongs to $A^2$ so $a\cdot 
c=b^2$ for some $b$ in $A$. Since the square function preserves 
randomness, we have that $b$ is random over 
$M, a$ by Lemma \ref{L:sym_random}. Set now $g=b\inv \cdot a= b\cdot c\inv$ and observe that the 
elements $c, 
g\cdot c$ and $g\cdot c\cdot g$  all belong to $A$. If the group is 
abelian, this is an actual $3$-AP as in the introduction. 
\begin{proof}
We may assume that $A$ is definable over a countable elementary substructure $M$, so it contains a weakly random type $p$ over $M$. Choose some weakly random element $g$ in $\Conn M$. By Fact 	\ref{F:Hr}, the partial type  $p\cdot g \cap p$ is weakly random, so 
	the set $A\cdot g\cap A$ has positive density. By Remark 
	\ref{R:posdens_random}, choose an element $a$ in $A\cdot g \cap 
	A$ random over $M,g$ and notice that $b=a\cdot g\inv$ lies in $A$ as well.  
 
 Since squaring preserves randomness, the element $a^2$ is also random over $M, g$ and hence so is $a\cdot b=a^2\cdot g\inv$ by Remark \ref{R:random_wide}. By Lemma \ref{L:ref}, the element $g$ is weakly random over $M, a\cdot b$, and hence $a^2=(a\cdot b)\cdot g$ is weakly random over $M,a\cdot b$. We deduce that $a$ is weakly random over $M,a\cdot b$, for squaring preserves randomness. Furthermore, multiplying on left by $(a\cdot b)\inv$ we conclude that $b^{-1}$, and hence $b$, is weakly random over $M,a\cdot b$.

Note that $b$ belongs to $A\inv \cdot (a\cdot b) \cap A $, so this intersection must have positive density. Corollary  \ref{C:measure_constant_coset} yields that the set $A\inv 
 \cdot a^2 \cap A$ has positive measure, for $a^2$ and $a\cdot b$ lie in the same coset modulo $\Conn M$.  Choose now some  some random 
 element $a_1$ in $A$ over $M, a$ with $a_1\inv\cdot a^2=a_2$ in 
 $A$. Remark 
 \ref{R:random_wide} and  Lemma \ref{L:sym_random} yield that the pair 
 $(a_1,a_2)$ is 
 random over $M$. Thus, the $M$-definable set 
\[ 
  \left\{ (x_1, x_2) \in A\times A \ | \ x_1\cdot x_2  \in A^2 \right\}
\]  
has positive $\mu_2$-measure, as desired. 
\end{proof}

\begin{question}
Consider a definably amenable pair $(G,X)$ such that the square 
function preserves randomness and let $M$ be a
countable elementary substructure $M$. Given an $M$-definable 
subset $A$ of $\sbgp X$ of positive density, does the $M$-definable 
set 
 \[ 
\left\{ (x_1, x_2) \in A\times A \ | \ x_2\cdot \ x_1\inv\cdot x_2  \in 
A \right\}
\] has positive $\mu_2$-density? Equivalently, is there a random 
pair $(a, b)$ in $A\times A$  over $M$ with $b\cdot a\inv\cdot b$ in 
$A$? 
\end{question}
Such a pair $(a, b)$ as above yields a $3$-AP in $A$ of the form $(a, 
a\cdot g, a\cdot g^2)$ with $g=a\inv b$. We do not currently know 
whether the above question has a positive answer, though it is the 
case for ultra-quasirandom groups, by the work of Tao \cite{tT13}. 

\begin{remark}\label{R:Eq_Roth_abelian}
The proof of Theorem \ref{T:Roth_defamen} in the abelian context yields immediately the existence of solutions to translation-invariant equations of the form \[ n_1 x_1+ \cdots n_m x_m= ky,\] whenever $k=\sum_{j=1}^m n_j$ and each of the maps $x\mapsto n_1 x$, $x\mapsto kx$ and $x\mapsto n'x$ preserves randomness, with $n'=\sum_{j=2}^m n_j$. That is, for every $M$-definable subset $A$ of $X$ of  positive density,  the set  \[ \mathcal{E}(A)= \left\{ (x_1, \cdots, x_m) \in A\times \stackrel{m}{\cdots}\times A \ | \ n_1 x_1+ \cdots n_m x_m= k c \text{ for some 
 c in }A \right\} \] has positive $\mu_m$-measure.  Indeed, choose $g$, $a$ and $b$ as in the proof of Theorem \ref{T:Roth_defamen}, so $g=a-b$. If we denote by $\ell A=\{\ell d\}_{d\in A}$, we will first show that  the set $n_1A + n'a\cap kA$ has positive density: By Corollary  \ref{C:measure_constant_coset}, we need only show that $n_1A + ka-n_1b \cap kA$ has positive density. Now, the element \[ ka-n_1b= n'a+ n_1g\] is random over $M, g$, since $a$ is random over $M, g$. So $g$, and thus $kg$, is weakly random over $M, ka-n_1b$ by Lemma \ref{L:ref}. Since \[ kg= ka- kb= (ka-n_1b)-n'b,\] we deduce that $-n'b$, and hence $n_1b$, is weakly random over $M, ka-n_1b$. Hence, the element $ka= n_1b + (ka-n_1b)$ is  weakly random over $M, ka-n_1b$ and belongs to $n_1A + ka-n_1b \cap kA$, as desired. 

 Choose now $a_2,\ldots, a_m$ realisations of $\tp(a/M)$ with $a_j$ weakly random over $M, a, g, a_2, \ldots, a_{j-1}$.  Hence, the differences $a_j-a$ all belong to $\Conn M$, by Fact \ref{F:Hr}. Corollary  \ref{C:measure_constant_coset} and the above paragraph yield that $n_1 A + \sum_{j=2}^m n_j a_j\cap k A$ has positive density, so choose an element $a_1$ in $A$  weakly random over $M, a_2,\ldots, a_m$ exemplifying that  the above intersection has positive density. The weakly random type $\tp(a_1,\ldots, a_m/M)$ contains the $M$-definable set  
$\mathcal{E}(A)$, as desired.  
\end{remark}

\section{Ultra-quasirandomness revisited}\label{S:ultra}

Given a definably amenable pair $(G, X)$ with $\sbgp X=G$, a  straight-forward application of compactness yields that  $X^{\odot n}=G$ for some natural number $n$ in $\N$, so $X$ generates $G$ in finitely many steps. Up to scaling the $\sigma$-finite measure, we may assume that $G=X$, so $\mu(G)=1$. This observation, together with  Examples \ref{E:amenable}(a) and \ref{E:fsg}, motivates the following notion. 

\begin{definition}\label{D:genPpal}
Let $(G, X)$ be a definably amenable pair with $X=G$. We say that the pair is {\em generically principal} if $G=\GO M$ for some elementary substructure $M$. 
\end{definition}
\noindent In an abuse of notation, we will simply say that the group $G$ is generically principal.

\begin{remark}\label{R:genppal_randomG00}
By \cite[Corollary 2.6]{MPP20}, a group $G$ is generically principal if and only if $G=\GO M$ for every elementary substructure $M$, so we may assume that $M$ is countable. 

In particular, a generically principal group contains trivially random 
elements concentrated in $\Conn M = G$ for every countable elementary 
substructure $M$.
\end{remark}

\begin{example}\label{E:genPPal_ex}
Three known classes of groups are generically principal: 
\begin{itemize}
\item Connected stable groups, such as every connected algebraic group over an algebraically closed field. 
    \item Simple definably compact groups definable in some o-minimal expansion of a real closed field, such as $\mathrm{PSL}_n(\mathcal R)$. 
    \item Ultra-quasirandom groups, introduced by Bergelson and Tao \cite{BT14}.
Let us briefly recall this notion. A finite group is \emph{$d$-quasirandom}, with $d\ge 1$, if all its non-trivial representations have degree at least $d$. An ultraproduct of finite groups $(G_n)_{n\in \mathbb N}$ with respect to a non-principal ultrafilter $\mathcal U$ is said to be {\em ultra-quasirandom} if for every integer $d\ge 1$, the set $\left\{ n\in \mathbb N \, | \, G_n \text{ is $d$-quasirandom} \right\}$
belongs to $\mathcal U$. 

The work of Gowers \cite[Theorem 3.3]{wG08} yields that every 
	definable subset $A$ of positive density of an ultra-quasirandom group $G(M)$  is not 
	{\em product-free}, 
	{\it i.e.} it contains a solution to the equation $xy=z$, and thus the 
	same holds in every elementary extension. Therefore, no weakly random type over an elementary substructure is product-free and thus $G=\GO N$ over any elementary 
	substructure $N$ by \cite[Corollary 2.6]{MPP20}, so ultra-quasirandom groups are generically principal. \end{itemize} 
\end{example}
Proposition \ref{P:stable_invariant} and its corollaries yield now a short proof of the result mentioned in the above paragraph. 

\begin{lemma}\label{L:genPPla_equiv}
 The following conditions are equivalent for a definably amenable pair $(G, G)$:
 \begin{enumerate}[(a)]
     \item The group $G$ is generically principal.
     \item Given two definable subsets $A$ and  $B$ of positive density, we have that $A\cdot B$ has measure $1$. In particular, whenever the definable subset $C$ has positive measure, so is $G=A\cdot B\cdot C$. 
     \item There is no definable product-free set of positive density.
 \end{enumerate}
\end{lemma}
\begin{proof}
For (a) $\Rightarrow$ (b):  Given  two subsets $A$ and $B$ of 
positive density  definable over 	some countable elementary 
substructure $M$,  we need only show that every weakly random 
every 	element $g$ lies in $	A\cdot 	B$. Now, 
Lemma \ref{L:inter_posdens} yields that there exists some random 
element $h$ over $M$ with $\mu(A\cap hB\inv)>0$. Corollary 
\ref{C:measure_constant_coset} gives that every element $g$ of $G$ 
weakly 
random over $M$ satisfies that $\mu(A\cap gB\inv)>0$ as well. So 
the definable set $A\cdot B$ has measure $1$, as desired. 

For the second assertion, given a definable set $C$ of positive density, let $g$ in $G$ be arbitrary. Now, 
\[
\mu(A\cdot B \cap g C\inv)= \mu(gC\inv) = \mu(C) >0,
\]
 so $g$ belongs to $A\cdot B\cdot C$, as desired. 
	
The implication (b) $\Rightarrow$ (c) is clear, taking $A$ and $B$ to be the same set. Thus, we are left to consider the implication (c) $\Rightarrow$ (a). Suppose that $G\ne \GO{M}$ for some countable elementary substructure $M$ and take a weakly random type $p$ in a non-trivial coset $\Cos{p}$  of $\GO{M}$. Note that $p\inv\cdot p \cdot p \subseteq \Cos{p}$. A standard compactness argument yields the existence of some $M$-definable set $A$ in $p$ such that $\mathrm{id}_G$ does not lie in $A\inv\cdot A \cdot A$, so $A$ is product-free. Since $p$ is weakly random, the definable subset $A$ has positive density. 
\end{proof}

The following result on weak mixing, already present as is in the work of Tao 
and Bergelson, 
was implicit in the work of Gowers \cite{wG08}. It will play a crucial role to 
study some instances of complete amalgamation of
equations in a group. 
\begin{cor}\label{C:density_prod}\textup{(}\cf\ 
	\cite[Lemma 33]{BT14}\textup{)}
Let $G$ be a generically principal group. Given two definable subsets $A$ and $B$ of positive density, 
 \[ \mu(A\cap gB)=\mu(A)\mu(B)\]    
	 for $\mu$-almost all elements $g$. 
\end{cor}	

\begin{proof}
As before, fix some countable elementary 
substructure $M$ such that both $A$ and $B$ are 
$M$-definable. We may assume that the measure $\mu$ is also definable 
over $M$. By Corollary \ref{C:measure_constant_coset}, set 
$\alpha=\mu(A\cap gB)$ for some (or equivalently, every) weakly 
random 
element $g$ over 
$M$.  Notice that $\alpha>0$ by Remark 
\ref{R:inter_posdens}. 
  
The subset \[ Z=\{x\in A\cdot B\inv \ |\ \mu(A\cap xB)=  
\alpha\} \] is type-definable over $M$ and contains all weakly 
random elements over $M$. Clearly, the measure $\mu(Z)\le \mu(AB\inv)$ and the latter equals $1$,  by Lemma  
\ref{L:genPPla_equiv}.   If  $\mu(Z)<\mu(A\cdot B\inv)$, there is an 
$M$-definable set $\tilde Z$ with $Z\subseteq \tilde Z \subseteq 
A\cdot B\inv$ such that $\mu(A\cdot B\inv \setminus \tilde Z)>0$. 
Thus, 
the 
set $A\cdot B\inv \setminus \tilde Z$ has positive density and it 
must 
contain a weakly random element over $M$, which gives the 
desired contradiction, so $\mu(Z)=\mu(A\cdot B\inv)=1$. 

Consider now the set
$ W=\left\{ (a, z) \in A\times A\cdot B\inv  \ | \  z=a\cdot b\inv 
\text{ for some $b$ in $B$}
\right\}$. Note $a$ belongs to $A\cap z\cdot B$ and $z$ lies in 
$aB\inv$ if $(a, z)$ belongs to $W$. If we denote by $\mu_2$ the 
normalized  measure in $G\times G$, an 
easy 
computation yields that 
\[
\mu_2(W) = \int_{z\in A\cdot B\inv} \mu(A\cap zB) = \alpha 
\mu(A\cdot B\inv)\stackrel{\ref{L:genPPla_equiv}}{=}\alpha.
\]
By Fubini, we also have that 
\[
\alpha=\mu_2(W) = \int_{a\in A} \mu(aB\inv) = \int_{a\in A} 
\mu(B) = \mu(A)\mu(B),
\]
which gives the desired  conclusion. 
\end{proof}

A standard translation using \L o\'s's theorem  yields the following finitary 
version:

\begin{cor}\textup{(}\cf\ 
	\cite[Lemma 5.1]{wG08} \& \cite[Proposition 3]{BT14}\textup{)}\label{C:ultra_finite}
For every positive $\delta$, $\epsilon$ and $\eta$ there is some integer 
$d=d(\delta, \epsilon, \eta)$ such that for every finite $d$-quasirandom group 
$G$ and subsets $A$ and $B$ of $G$ of density at least $\delta$, we have that 
\[\left| \left\{x \in G \ \big| \ |A\cap xB| |G| < (1-\eta) |A||B| \right\} 
\right|< 
\epsilon |G|.\]
\end{cor}
\begin{proof}
 Assume for a contradiction  that 
	the statement does not hold, so
	there are some fixed positive numbers $\delta$, $\epsilon$ and  $\eta$ such that for each natural number $d$ we 
	find two subsets $A_{d}$  $B_{d}$ of a finite $d$-quasirandom group 
	$G_{d}$, each of density at least $\delta$, such that the cardinality of the subset 
	\[
	\mathcal{X}(G_d)= \left\{x \in G_d \ | \ |A_d\cap xB_d| |G_d| < (1-\eta) |A_d||B_d|\right\}\]  is at least 
	 $ \epsilon|G_{d}|$. 
	
	Following the approach of Example \ref{E:amenable}(a), we consider 
	a suitable
	expansion $\LL$ of the language of groups and regard each group $G_{d}$ 
	as an
	$\LL$-structure $N_{d}$. Choose a non-principal ultrafilter $\mathcal U$ on
	$\N$ and consider the
	ultraproduct $N=\prod_{\mathcal U} N_{d}$. The language $\LL$ is chosen in
	such a way that the sets $A=\prod_{\mathcal U}	A_{d}$ and 
	$B=\prod_{\mathcal U}  B_{d}$  are $\LL$-definable in the ultra-quasirandom group
	$G=\prod_{\mathcal U} G_{d}$. Furthermore, the normalised 
	counting measure on $G_{d}$ induces a definable 
	Keisler measure $\mu$ on $G$, taking the standard
	part of the ultralimit. By Corollary \ref{C:density_prod}, for $\mu$-almost all $g$ in $G$, we have $
	\mu(A\cap gB) = \mu(A)\mu(B)$. Hence, the type-definable set \[ \Sigma =
	\left\{x \in G \ \big| \    \mu(A\cap xB) \le  (1-\eta)\mu(A)\mu(B)    \right\}
	\] does not contain any weakly random type. By compactness, it is contained in a definable set $W$  whose density is $0$, and in particular its density is strictly less than the fixed value $\epsilon$. Since every element in the ultraproduct of the sets $\mathcal X(G_d)$ clearly lies in $\Sigma$, we conclude by  \L o\'s's theorem that $|\mathcal{X}(G_d)| \le |W(G_d)| < \epsilon |G_d|$ for infinitely many $d$'s, which yields the desired contradiction.
	\end{proof}

The following result is a verbatim adaption of \cite[Theorem 5.3]{wG08} and may be seen as a first attempt to solve complete 
amalgamation problems whilst restricting the conditions to 
those given by  products. 

\begin{theorem}\label{T:completemixing}
Fix a natural number $n\geq 2$.  For each non-empty subset $F$ of $\{1,\ldots, 
n\}$, let $A_F$ be a subset of positive density of the generically principal group $G$. The set \[ \mathcal X_n=\left\{(a_1,\ldots, a_n) \in G^n \ 
| \ 
a_F
\in A_F \text{ for all } \emptyset\neq F\subseteq\{1,\ldots, n\} \right\}\]
has measure $\prod_{F} \mu(A_F)$ with respect to the 
measure $\mu_n$ on $G^n$, where $a_F$ 
stands for the product of all $a_i$ with $i$ in $F$ written with 
the 
indices in increasing order. 
\end{theorem}

\begin{proof}
We reproduce Gower's proof of \cite[Theorem 5.3]{wG08} and proceed by induction 
on $n$. For $n=2$, set $B=A_{\{2\}}$ and $C=A_{\{1,2\}}$. A pair $(a,b)$ lies in $\mathcal X_2$ if and only if $a$ belongs to  $A_{\{1\}}$ and $b$ to $B\cap a\inv C$. 
Thus \[ \mu_2(\mathcal X_2) = \int_{A_{\{1\}}} \mu(B\cap a\inv C) \ d\mu 
\stackrel{\textrm{Cor. } \ref{C:density_prod}}{=} \mu(B) \mu(C) \mu(A_{\{1\}}),\] as 
desired.  For the general case, for any $a$ in $A_{\{1\}}$, set $B_{F_1}(a)= 
A_{F_1}\cap 
a\inv A_{1,F_1}$, for $\emptyset\neq F_1\subseteq \{2,\ldots, n\}$. Corollary 
\ref{C:density_prod} yields that $\mu(B_{F_1}(a)) =\mu(A_{F_1})\mu(A_{1,F_1})$ 
for 
$\mu$-almost all $a$ in $A_{\{1\}}$. A tuple $(a_1,\ldots, a_n)$ in $G^n$ belongs to 
$\mathcal 
X_n$ if and only if the first coordinate $a_1$ lies in $A_{\{1\}}$ and the tuple $(a_2,\ldots, a_n)$ belongs to 
\[ \mathcal 
X_{n-1}(a_1)=\left\{(x_2,\ldots, x_n) \in G^{n-1} \ | \ 
x_{F_1}  
\in B_{F_1}(a_1) \text{ for all } \emptyset\neq F_1\subseteq\{2,\ldots, n\} 
\right\}.\]   By induction, the set $\mathcal 
X_{n-1}(a)$ has constant $\mu_{n-1}$-measure $\prod_{F_1} 
\mu(A_{F_1})\mu(A_{1,F_1})$, where  
$F_1$ now runs through all non-empty subsets of $\{2,\ldots, n\}$. Thus 
\[ \mu_n(\mathcal X_n) = \int_{A_1} \mu_{n-1}(\mathcal X_{n-1}(a_1)) \ d\mu = 
 \mu(A_1)  \prod_{F_1} \mu(A_{F_1})\mu(A_{1,F_1})= \prod_F \mu(A_F),\]
 which yields the desired result.
\end{proof}
A standard translation using \L o\'s's theorem (we refer to the proof of Corollary \ref{C:ultra_finite} to avoid repetitions)  
yields the following finitary 
version, which was already present in a quantitative form 
in Gowers's work \cite{wG08}.

\begin{cor}\label{C:gowers_mixing}\textup{(}\cf\ 
	\cite[Theorem 5.3]{wG08}\textup{)}
Fix a natural number $n\geq 2$. For every $\emptyset\neq F\subseteq \{1,\ldots, 
n\}$ let $\delta_F>0$ be given. For every $\eta>0$ there is some integer 
$d=d(n, \delta_F, \eta)$ such that for every finite $d$-quasirandom group 
$G$ and subsets $A_F$ of $G$ of density at least $\delta_F$, we have that 
\[\left| \mathcal X_n \right| \ge \frac{1-\eta}{|G|^{2^n-1-n}} \prod_{F} 
|A_F|,\]
where  $\mathcal X_n$ is defined as in Theorem \ref{T:completemixing} with 
respect to the group $G$. 
\end{cor}

The above corollary yields in particular that  
\[ \left| \left\{(a, b, c) 
\in A\times B \times C \ | \ ab=c \right\}\right|> \frac{1-\eta}{|G|} |A||B||C| 
\] as first proved by Gowers \cite[Theorem 
3.3]{wG08}, which implies that the number of such triples is a proportion 
(uniformly on the densities and $\eta$) of $|G|^2$. 

To conclude this section we answer affirmatively the question in the introduction for generically principal groups, whenever all the types are based over a common countable elementary substructure.

\begin{theorem}\label{T:GenPrinMainQuestion}
Fix a natural number $n\geq 2$ and a countable elementary 
substructure $M$ of the generically principal definably amenable 
pair $(G, X)$.  For each non-empty subset $F$ of $\{1,\ldots, 
n\}$, let $p_F$ be a weakly random type over $M$. There exists a weakly random $n$-tuple
$(a_1,\ldots, a_n)$ in $G^n$ such that $a_F$ realises $p_F$ for all $\emptyset\neq F\subseteq\{1,\ldots, n\}$, where $a_F$ 
stands for the product of all $a_i$ with $i$ in $F$ written with 
the 
indices in increasing order.  
\end{theorem}
\begin{proof}
Since $M$ is countable, enumerate all the formulae occurring in each type $p_F$ in a decreasing way, that is, write $p_F=\{A_{F,k}\}_{k\in \N}$ with $A_{F,k+1} \subseteq A_{F,k}$ for every natural number $k$. 
We want to show that the set 
\[ 
\mathcal X_{n}=\left\{(x_1,\ldots, x_n) \in G^n \ 
| \ p_F(x_F)  \text{ for all } \emptyset\neq F\subseteq\{1,\ldots, n\} \right\}
\]
is weakly random over $M$, that is, we need to prove that the partial type 
\[ \{ \neg\psi(x_1,\ldots, x_n)\}_{\psi \in \Sigma}  \cup \{ x_F \in A_{F,k}\}_{\substack{F\in \mathcal P \\ k\in \N  }} \] is consistent,
where $\mathcal P =\mathcal P(\{1,\ldots,n\})\setminus\{\emptyset\}$ and $\Sigma$ is the set of $\LL_{M}$-formulae of $\mu_n$-measure $0$. 
By compactness, since the subsets $A_{F,k}$ are enumerated decreasingly, we need only consider a finite subset of the above partial type where the level $k_0$ is the same for each of the subsets $A_{F,k_0}$ of positive density. By  Theorem \ref{T:completemixing} the set 
\[ \mathcal X_{n,k_0}=\left\{(a_1,\ldots, a_n) \in G^n \ 
| \ 
a_F
\in A_{F,k_0} \text{ for all } \emptyset\neq F\subseteq\{1,\ldots, n\} \right\}\]
has $\mu_n$-measure $\prod_{F} \mu(A_{F, k_0})>0$, so we conclude the desired result. 
\end{proof}

\section{Local ultra-quasirandomness}\label{S:localultra}
In this final section, we will adapt some of the ideas present in Section \ref{S:ultra} to arbitrary finite groups. 

Theorem \ref{T:pqr} holds  in any definably 
amenable pair for  \emph{almost all} three weakly random types, whenever their cosets modulo $\GO M$ are product-compatible. Thus, it yields
asymptotic information for subsets of positive 
density in arbitrary finite groups satisfying certain regularity 
conditions, which force that in the ultraproduct some 
completions are in a suitable position to apply our main Theorem 
\ref{T:pqr}. We will present two examples of such regularity
notions. Our intuition behind these notions is purely 
model-theoretic and we ignore whether it is meaningful from a combinatorial perspective. We would like to express our gratitude to Julia Wolf (and indirectly to Tom Sanders) for pointing out that our previous definition of principal subsets did not extend to the abelian case.  

\begin{definition}\label{D:ppal}
Let $A$ be a definable subset of $\langle X\rangle$ of positive density in a definably amenable pair $(G,X)$. We say that $A$ is \emph{principal} over the parameter set $B$ if 
		\[ \mu(A \cap (Y\cdot Y) ) >0\] whenever $Y$ is a $B$-definable neighborhood of the identity (that is, the set $Y$ is symmetric and contains the identity) such that finitely many left translates of $Y$ cover $A\cdot A\inv\cdot A\cdot A\inv$. 
		
		Analogously, we say that $A$ is \emph{hereditarily principal over the parameter set $B$} if all of its $B$-definable subsets of positive density are principal.
\end{definition}

\begin{remark}\label{R:ultra_ppal}
Let $A$ be a definable subset of $\sbgp X$ of positive density of a  definably amenable pair $(G,X)$ such that  $\mu(A \cap (Y\cdot Y) ) = \mu(A)$, whenever $Y$ is a definable neighborhood of the identity which covers $A\cdot A\inv\cdot A\cdot A\inv$ with finitely many left translates. Then the set $A$ is hereditarily principal over any subset of parameters. 
\end{remark}
\begin{proof} Let $A_0$ be a definable subset of $A$ of positive measure. Notice that there is a maximal finite subset $F$ of $(A A\inv)^2$ with the property that $\mu(xA_0\cap yA_0)=0$ for any two distinct $x$ and $y$ in $F$.  In particular, the set $(A A\inv)^2$ is contained in $ F\cdot A_0\cdot A_0\inv$. Thus, any  
definable neighborhood $Y$ of the identity such that finitely many left translates of cover $A_0A_0\inv A_0 A_0\inv$ also cover $AA\inv A A\inv$, so $\mu(A \cap (Y\cdot Y) ) = \mu(A)$ by assumption on $A$. Hence $\mu(A_0\cap (Y Y))=\mu(A_0)>0$, as desired. 
\end{proof}
\begin{example}
If $G$ is generically principal, every definable subset $A$ of positive density is hereditarily principal over any  parameter set: Indeed, Lemma  \ref{L:genPPla_equiv} yields that $G=A\cdot A\inv\cdot A\cdot A\inv$. Therefore, finitely many translates of the neighborhood $Y$ must cover $G$, so $Y$ has positive measure and hence $\mu(Y\cdot Y)=1$ by Lemma \ref{L:genPPla_equiv}. 

By the previous remark, the definable subset $A$ satisfies that $\mu(A \cap (Y\cdot Y) ) = \mu(A)$, so $A$ is hereditarily principal over any subset of parameters. 
\end{example}

\begin{example}
Fix some enumeration $(q_n)_{n\in \N}$ of all the primes and consider the family of groups  $(G_n=  \mathrm{PSL}_2(q_n)\times \Z_2)_{n\in \N}$, each equipped with  the  distinguished subset $X_n=\mathrm{PSL}_2(q_n)\times\{\bar 0\}$. This family produces a definably amenable pair $(G,X)$, as  in the Example \ref{E:amenable}.  Note that \[ G=\mathrm{PSL}_2(\mathbb F)\times \Z_2 \text{ and } X=\mathrm{PSL}_2(\mathbb F)\times \{\bar 0\}\] for some infinite (pseudofinite) field $\mathbb F$. Over any elementary substructure $M$ we have that $\GO M$ equals the simple group $X=\mathrm{PSL}_2(\mathbb F) \times \{\bar 0\}$, which is clearly definable. The definable subset $G$ is clearly principal yet not hereditarily principal, for the dense subset  $X\cdot (0_{\mathrm{PSL}_2(\mathbb F)}, \bar 1)$ does not intersect $X=\GO M$.
\end{example}

\begin{lemma}\label{L:ppal_type_ppal}
Let $M$ be a countable elementary substructure of a definably amenable pair $(G, X)$.
\begin{enumerate}[(a)]
    \item Principal definable sets over $M$ contain weakly random 
    principal types in $S_\mu(M)$, that is, types concentrated in 
    $\Conn M$.
    \item Every weakly random type over $M$ containing a 
    hereditarily principal definable set is principal. 
\end{enumerate}
\end{lemma}
\begin{proof}
For (a), assume that the $M$-definable set $A$ is principal over the model $M$. Note that we can write the type-definable subgroup $\Conn{M}$ as a countable intersection 
	\[
	\Conn{M} = \bigcap_{i\in \N} V_i,
	\]
	where the decreasing chain $(V_i)_{i\in \N}$ consists of 
	$M$-definable neighborhoods of the identity such that 
	$V_{i+1}\cdot V_{i+1}\subseteq V_i$ for all $i$ in $\N$. Since 
	$\Conn{M}$ has bounded index in the subgroup $\sbgp X$,  compactness yields that 	finitely many translates of each $V_i$ cover the subset $A\cdot 
	A\inv \cdot A \cdot A\inv$ (yet the number of 
	translates possibly depends on $i$). Hence, the type-definable 
	subset $A \cap \Conn M$ is weakly random, since $A$ is principal, so 
	$A$ contains a weakly random type concentrated in $\Conn M$, as desired.

 For (b), suppose that the $M$-definable set $A$ is hereditarily 
 principal yet it contains a weakly random type $q$ which does not 
 concentrate on $\Conn M=\bigcap_{i\in \N} V_i$, with the same notation as above. By 
 compactness, this implies the existence of some $i$ in $\N$ and 
 some $M$-definable subset $A_0$ of $A$ of positive density with 
 $A_0\cap V_i=\emptyset$. The subset $A_0\cap (V_{i+1}\cdot 
 V_{i+1})$ has in particular measure $0$, so $A_0$ is not principal, 
 contradicting our assumption on $A$. 
\end{proof}

\begin{prop}\label{P:ppal}
Consider a  subset $A$ of positive density definable over a countable elementary substructure $M$ of a sufficiently 
saturated definably amenable pair $(G,X)$. If $A$ contains a weakly 
random type $p$ concentrated in $\Conn M$, then the subset 
\[  \left\{(a, b) 
\in A\times A \ | \ a\cdot b \in A \right\} \] has positive $\mu_2$-measure. 

In particular, if $A$ is principal, then the above set of pairs  has positive 
$\mu_2$-measure.
\end{prop}
\noindent Notice that the definable set $A$ above cannot be 
product-free, for the equation $x\cdot y =z$ has a solution in $A$. 
\begin{proof}
The proof is an immediate  application of Fact \ref{F:Hr}: Indeed, for 
every realization $a$ of $p$, the partial type $p\cap a\inv\cdot p$ is 
weakly random (for the weakly random element $a$ over $M$ belongs to 
$\Conn M$), so choose a weakly random element $b$ over $M, a$ 
realizing $p$ such that $a\cdot b$ does it as well. By Lemma \ref{L:trans}, we obtain a weakly 
random type $\tp(a, b/M)$ with all three elements $a, b$ and 
$a\cdot b$ in $A$, which yields immediately the desired result. 
\end{proof}

Proposition \ref{P:ppal} resonates with work of Schur \cite[Hilfssatz]{iS16} on the number of monochromatic triples  $(x, y, x\cdot y)$ in any finite coloring (or cover) of the natural numbers $1,\ldots, N$, for $N$ sufficiently 
large.  In fact, by a standard  application of \L 
 o\'s's theorem, the above argument yields a non-quantitative version of the following result of Sanders \cite[Theorem 
 1.1]{tS19}: 
\medskip

\noindent \textit{For every natural number $k\ge 1$ there is some 
$\eta=\eta(k)>0$ with the following 
property: Given any coloring on a finite group $G$ with $k$ many colors 
$A_1,\ldots, A_k$, there exists some color $A_j$, with $1\le j\le 
k$, such that}
\[
\left| \left\{(a, b, c) 
\in A_j\times A_j \times A_j \ | \ a\cdot b =c \right\} \right| \ge 
\eta|G|^2.
\]
Motivated by Gowers's result  \cite[Theorem 5.3]{wG08} for (ultra-)quasi\-random groups, we will now provide a weaker version of it, taking all $A_F$'s to be the same subset $A$, for $\emptyset \neq F\subseteq \{1,\ldots, n\}$ as in Corollary \ref{C:gowers_mixing}.

\begin{cor}\label{C:ppal}
	In a sufficiently saturated definably amenable pair $(G,X)$ with associated measure $\mu$, consider a definable subset $A$ of $X$ of positive density which is hereditarily principal over the parameter set $G$ itself. For every countable elementary substructure $M$ of $(G, X)$ such that both the measure and the sets $A$ are  $M$-definable,  there is a tuple $(a_1,\ldots, a_n)$ in $G^n$  weakly random over
		$M$ 
		such that the product $a_F$ (as in Theorem \ref{T:GenPrinMainQuestion}) lies in $A$ for every subset $F$ as 
		above. 
\end{cor}
An inspection of the proof shows that it suffices if the definable set 
$A$ is hereditarily principal over $N$, where $N$ is an 
$\aleph_1$-saturated elementary substructure of $(G, X)$ 
containing $M$. This is not surprising, since an easy compacteness 
argument shows that a set $A$ which is hereditarily principal over 
an $\aleph_1$-saturated elementary substructure $N$ of $(G, X)$ 
must be hereditarily principal over the parameter set $G$ itself. 
\begin{proof}
	We proceed by induction on the natural number $n$. Since both the base case $n=3$ and the induction step have similar proofs, we will assume that the statement of the Corollary has already been shown for $n-1$.

 The set $A$ is principal, so it contains a weakly random type concentrated in $\Conn M$, by Lemma \ref{L:ppal_type_ppal} (a). As in the proof of Proposition \ref{P:ppal}, there is a weakly random 
 element $a_1$ in $A$ over $M$ such that $A'=A\cap a_1\inv\cdot A$ has 
 positive density. Notice that $A'$ is no longer definable over $M$, 
 yet it is again hereditarily principal over the parameter set $G$. By 
 Downwards Löwenheim-Skolem, choose some countable elementary 
 substructure $M_1$ of $(G,X)$ containing $M\cup\{a_1\}$. By 
 induction, there is a tuple $(a_2,\ldots, a_n)$, 
	weakly 
	random over $M_1$, such that each product $a_{F_1}$ 
	lies in $A'$ for every subset $\emptyset \neq 
	F_1\subseteq\{2,\ldots, n\}$.  For $n=3$,  we obtain such a tuple by applying Proposition \ref{P:ppal} to the principal $M_1$-definable set $A'$.
	
	Lemma \ref{L:trans} yields now that the 
	tuple $(a_1,\ldots, a_n)$ is weakly random over $M$. By construction, 
	the 
	product $a_{F}$ 
	lies in $A$ for every subset $\emptyset \neq 
	F\subseteq\{1,\ldots, n\}$, as desired. 
\end{proof}

Motivated by the above result, we isolate  a  
particular instance of  a complete amalgamation problem (\cf the 
question in the  introduction).  
\begin{question}
Let $M$ be a countable elementary substructure of  a sufficiently 
saturated definably amenable pair $(G,X)$ and $p$ be a weakly 
random  type in $\Conn {M}$. Given a natural number $n$, is 
there a tuple $(a_1,\ldots, a_n)$ in $G^n$ weakly random over $M$ such that 
$a_F$ realizes $p$ for all  $\emptyset\neq F\subseteq\{1,\ldots, 
n\}$, where $a_F$ stands for the product, enumerated in an 
increasing order, of all $a_i$ with $i$ in $F$?
\end{question} 
At the moment of writing, we do not have a solid guess what the 
answer to the  above question will be. Following the lines of the proof of Corollary \ref{C:ppal}, the above question 
would have a positive answer if the following statement is true: 

\medskip
\emph{Let $p=\tp(a/M_0)$ be a weakly random type in $\Conn{M_0}$, where $M_0$ is a countable elementary substructure of a saturated definably amenable pair $(G,X)$. Then there are an elementary substructure $M_1$ containing $M_0\cup\{a\}$ and a weakly random type $q$ in $\Conn{M_1}$ extending $p\cap a\inv \cdot p$}
\medskip

Nonetheless, if  the question 
could be positively answered, it 
would imply by a standard compactness argument a finitary version of  Hindman's Theorem 
\cite{nH74}.  
\begin{remark}\label{R:IP}
	If the above question has a positive answer, then for every natural 
	numbers $k$ and $n$ there is some constant
	$\eta=\eta(k,n)>0$ such that in any coloring on a finite group 
	$G$ with $k$ many colors 
	$A_1,\ldots, A_k$, there exists some color $A_j$, with $1\le j\le 
	k$ such that
	\[\left| \left\{(a_1,\ldots, a_n) \in G^n \ | \ 
	a_F
	\in A_j \text{ for all } \emptyset\neq F\subseteq\{1,\ldots, n\} \right\} 
	\right| \ge \eta|G|^{n},\]
	where $a_F$ stands for the product, enumerated in an 
	increasing order, of all $a_i$ with $i$ in $F$.
\end{remark}

We can now state the finitary versions of principal sets to provide finitary analogs of Proposition \ref{P:ppal}
 and Corollary \ref{C:ppal}.
 \begin{definition}\label{D:ppal_fte}
Fix $\epsilon>0$ and $k$ in $\N$.  	A finite subset $A$ of a group $G$ is \emph{$(k,\epsilon)$-principal} if 
		\[ |A \cap (Y\cdot Y)| \geq \epsilon |A|\] whenever $Y$ is a neighborhood of the identity (that is, the set $Y$ is symmetric and contains the identity) such that $k$ many left translates (or equivalently, right translates) of $Y$ cover $A\cdot A\inv\cdot A\cdot A\inv$. 

\noindent We shall say that the finite subset $A$ is \emph{hereditarily $(k,\epsilon)$-principal up to $\rho$} if all its subset of relative density  at least $\rho$ (in $A$) are $(k,\epsilon)$-principal.
\end{definition}

\begin{example}
Consider the finite group $G=\Z_n\times \Z_2$. The set $G$ is clearly $(k,1/k)$-principal for every natural number $k\ne 0$, yet it is not hereditarily $(2,1/k)$-principal up to $1/2$ for any $k\ne 0$, for the subset  $A=\Z_n\times \{\bar 1\}$ does not intersect $Y=\Z_n\times \{\bar 0\}$, which covers $G$ in $2$ steps.
\end{example}

\begin{example}
Given a subset $A$ of a finite group $G$ of density at least $\epsilon$, the symmetric set $AA\inv$ is $(k,\epsilon/k)$-principal. Indeed, if $Y$ is a given neighborhood of the identity such that $k$ many right translates of $Y$ cover $(AA\inv)^4$, then there exists some $c$ in $G$ such that $|Ac\cap Y|\ge |A|/k$ and so $|AA\inv \cap YY|\ge \epsilon|AA\inv|/k$, since $(Ac\cap Y)(Ac\cap Y)\inv \subseteq AA\inv \cap YY$. 
\end{example}
\begin{cor}\label{C:ppaln=2}
    Let $K>0$ and $\delta>0$ be given real numbers. There are real values $\epsilon=\epsilon(K, \delta)>0$ and
	$\eta=\eta(K, \delta)>0$ as well as a natural number $k=k(K, \delta)$ such that for every group
	$G$  and a finite subset $X$ of $G$ of tripling at most $K$ together with a $(k,\epsilon)$-principal subset $A$ of $X$ of relative density at 
	least $\delta$ with respect to $X$, the collection of triples  
	\[ \left\{(a,b) \in A\times A \ | \ 
	a\cdot b \in A \right\} \] 
	has size at least  $\eta|X|^{2}$.
\end{cor}
\begin{proof}
Assume for a contradiction that 
	the statement does not hold. Negating quantifiers there are positive constants $K$ and $\delta$ such that for each triple $\bar\ell =(k, n,m)$ of natural numbers there exists a group $G_{\bar\ell}$ and a finite subset $X_{\bar\ell}$ of $G_{\bar\ell}$ of tripling at most $K$ as well as a $(k,1/n)$-principal subset $A_{\bar\ell}$ of $X_{\bar\ell}$ of relative density at least $\delta$  such that the cardinality of the subset 
	\[
	\mathcal{Y}(G_{\bar\ell})= \left\{(x,y) \in A_{\bar\ell}\times A_{\bar\ell}  \ | \	x\cdot y \in  A_{\bar\ell} \right\}\]  is bounded above by 
	 $|X_{\bar\ell}|^2/m$. 
	
	Following the approach of the Example \ref{E:amenable}\,(b), we consider 
	a suitable countable expansion $\LL$ of the language of groups and regard each such group $G_{\bar\ell}$, with $\bar\ell$ of the form $(k,k,k)$, 
	as an
	$\LL$-structure $N_{\bar\ell}$ in such a way that $\LL$ contains predicates for $X_{\bar\ell}$ and $A_{\bar\ell}$. Identify now the set of such triples $(k,k,k)$ with the natural numbers in a natural way and choose a non-principal ultrafilter $\mathcal U$ on
	$\N$. Consider the
	ultraproduct $N=\prod_{\mathcal U} N_{\bar\ell}$. As outlined in the Example \ref{E:amenable}, this construction gives rise to a definable amenable pair $(G,X)$ with respect to a measure $\mu$ equipped with an $\emptyset$-definable subset $A$ of $X$ of positive density (at least $\delta$) such that $\mu_2(\mathcal Y(G)) =0$.  Notice that $A$ is now principal over the parameter set $N$, by \L o\'s's theorem. 
	
	Fix a countable elementary substructure $M$ of $N$. 
  By Proposition \ref{P:ppal},  the set
	\[
	\mathcal Y(G) = \left\{ (x,y)\in A\times A \, | \,  x\cdot y\in A \right\}
	\]
	has positive density with respect to $\mu_2$, which contradicts the ultraproduct construction.
\end{proof}
The proof of the next result follows from Corollary \ref{C:ppal} along the same lines as Corollary \ref{C:ppaln=2} by a standard
	ultraproduct argument using \L o\'s's theorem (and implicitly that a non-principal ultraproduct of finite sets is $\aleph_1$-saturated). 
\begin{cor}\label{C:ppaln_ge3}
 For a natural number $n\geq 3$, let  real numbers $K>0$ and $\delta_F>0$,  for $\emptyset\neq F\subseteq \{1,\ldots, 
	n\}$ be given. There are $\epsilon=\epsilon(n, K, \delta_F)>0$,  $\rho=\rho(n, K, \delta_F)$ and 
	$\eta=\eta(n,K, \delta_F)>0$ as well as a natural number $k=k(n, K, \delta_F)$ such that for every group 
	$G$  and a finite subset $X$ of $G$ of tripling at most $K$ together with a subset $A$ of $X$ of relative density at 
	least $\delta$, whenever 
	\[\left| \left\{(a_1,\ldots, a_n) \in G^n \ | \ 
	a_F
	\in A \text{ for all } \emptyset\neq F\subseteq\{1,\ldots, n\} \right\} 
	\right| < \eta|X|^{n},\]
	where $a_F$ stands for the product, enumerated in an 
	increasing order, of all $a_i$'s with $i$ in $F$, then $A$ cannot be hereditarily  $(k,\epsilon)$-principal up to $\rho$. 
\end{cor}
In order to extend Proposition \ref{P:ppal} to pairs $(a, b)$ in  the cartesian product $A\times B$ with $a\cdot b$ in $C$, we will introduce a new notion, which we will refer to as compatibility for certain subsets in a definably amenable pair.

\begin{definition}\label{D:compatible_infte}
Let $A$, $B$ and $C$ be subsets of $\langle X\rangle$ of positive density in a definably amenable pair $(G,X)$, all three definable over the countable elementary substructure $M$. We say that $A$ and $B$  are \emph{compatible} with respect to $C$ over $M$ if there exists a random pair $(a, b)$ in $A\times B$ over $M$ such that $a\cdot b$ lies in the same coset modulo  $\Conn M$ as some element $c$ of $C$ which is weakly random over $M$.
\end{definition}		

It is clear that every two definable subsets $A$ and $B$ of positive density in a generically principal group $G$ are compatible with respect to any subset $C$ of positive density over any countable elementary substructure $M$ containing the parameters of definition of all three sets. More generally, we have the following observation. 

\begin{remark}\label{R:subsets_comp} Given three definable subsets $A, B$ and $C$ of positive density at least $\delta >0$ in a definably amenable pair $(G,X)$, all three defined over a countable elementary substructure $M$, every weakly random type of $\Conn M$ is contained in 
\[ 
A\cdot A\inv \cap B\cdot B\inv \cap  C\cdot C\inv,
\]  by Fact \ref{F:Hr}.  Hence, the $M$-definable set 
\[
\left\{(x, y) \in (A\cdot A\inv) \times (B\cdot B\inv) \ | \ x\cdot y \in C\cdot C\inv \right\} 
\] contains a pair $(a_1, b_1)$ with $a_1$ and $b_1$ both in $\Conn M$ weakly random over $M$ such that $b_1$ is weakly random over $M,a_1$. Hence, the above set has positive density,  so there exists a random pair $(a,b)$ in $AA\inv\times BB\inv$ over $M$ such that $a\cdot b$ belongs to $CC\inv$. Since $a\cdot b$ is (weakly) random over $M$, we deduce that $A\cdot A\inv$ and $B\cdot B\inv$ are compatible with respect to $C\cdot C\inv$.
\end{remark}

\begin{lemma}\label{L:ultra_ppal}
Let $A$, $B$ and $C$ be subsets of $\langle X\rangle$ of positive density in a definably amenable pair $(G,X)$, all three definable over the countable elementary substructure $M$. 
\begin{enumerate}[(a)]
    \item If for some element $g$ in $\Conn M$ weakly random over $M$,  the definable subset 
    \[
    Z_g=\{(a, b) \in  A\times B  \ | \ a\cdot b \in C\cdot g \}    \] has positive $\mu_2$-measure, 
     then $A$ and $B$ are compatible with respect to $C$ over $M$. 
    \item If $A$ and $B$ are  compatible with respect to $C$ over $M$, then the $M$-definable set 
\[
\{(a, b) \in A\times B \ | \ a\cdot b \in C \} 
\] 
has positive $\mu_2$-measure. 
\end{enumerate}
\end{lemma}
\begin{proof} For (a), given a weakly random element $g$ in $\Conn M$ suppose that the definable set $Z_g$ has positive density. By Remark \ref{R:posdens_random}, choose some $(a, b)$ in $Z_g$ random over $M, g$, so the element $c=a\cdot b\cdot g$ is again random over $M$ by Remark \ref{R:random_wide} and Lemma \ref{L:sym_random}. This  immediately yields that $A$ and $B$ are compatible with respect to $C$ over $M$.

For (b), suppose that $A$ and $B$ are compatible with respect to $C$ over $M$, so by definition, there is a random pair $(a, b)$ in  $A\times B$ over $M$ such that $a\cdot b$ lies in the same coset of $\Conn M$ as some element  $c$ in $C$ whose type over $M$ is weakly random.
By Lemma \ref{L:sym_random}, the pair $(a\inv, a\cdot b)$ is a random pair over $M$, so the definable set $A\inv \cdot (a\cdot b)\cap B$ has positive measure, for it belongs to the weakly random type $\tp(b/M, a\cdot b)$. By Corollary \ref{C:measure_constant_coset}, we deduce that $A\inv \cdot c \cap B$  has positive measure, so choose $b_1$ in $B$ weakly  random over $M, c$ such that $c=a_1\cdot b_1$. In particular, the $M$-definable set 
\[
\{(y, z) \in B \times C \ | \ z\cdot y\inv \in A \} 
\] has positive $\mu_2$-measure and so it contains a random pair $(b_2,c_2)$ over $M$. The pair $(a_2, b_2)$ of $A\times B$, with $a_2=c_2\cdot b_2\inv$ is again random over $M$ by Lemma \ref{L:sym_random} and satisfies that $a_2\cdot b_2$ belongs to $C$, as desired. 
\end{proof}
\begin{remark}
If the definable set $A$ has positive density and the pair $(A,A)$ is 
compatibly with respect to $A$ over a countable elementary 
substructure $M$, then $A$ is not product-free (cf. the 
corresponding comment after Proposition \ref{P:ppal}). On the other 
hand, is it the case that every principal definable set yields a 
compatible pair? Or are the two notions unrelated, even if they 
provide the same positive answer?
\end{remark}

Lemma \ref{L:ultra_ppal} yields a sufficient condition to ensure that the corresponding ultraproducts of finite subsets will be compatible. We have several candidates of finitary versions of compatibility, which will allow us to obtain a local version of \cite[Theorem 5.3]{wG08} to count the 
number of pairs in $A\times B$ such that the product $a\cdot b$ lies in the subset $C$ of positive density, all within a finite subset of small tripling. However,  it is unclear to us how combinatorially relevant our tentative definitions are, so we would rather leave the ultraproduct formulation as an open question: Is there a meaningful combinatorial definition (akin to Definition \ref{D:ppal_fte}) of when two finite sets $A$ and $B$ are compatible with respect to the finite set $C$?


\begin{thebibliography}{99}

\bibitem{BT14} V. Bergelson and T. Tao, \emph{Multiple recurrence in
    quasirandom groups}, Geom. Funct. Anal. {\bf 24}, (2014), 1--48.
    
\bibitem{vB07} V. I. Bogachev, \emph{Measure theory},  {\bf II}, (2007), Springer, Berlin,  xiii+575 p.

\bibitem{CGH23} G. Conant, K. Gannon and James Hanson, \emph{
Keisler measures in the wild}, J. Model Theory {\bf 2}, (2023), 1--67.

\bibitem{CPT20} G. Conant, A. Pillay and C. Terry, 
\emph{A group version of stable regularity}, Math. Proc. Cambridge Phil. Soc., {\bf 168}, (2020), 405--413. 


\bibitem{CS10}  E. Croot and O. Sisask, \emph{A probabilistic technique for finding almost-periods of convolutions},  Geom.
Funct. Anal. {\bf 20}, (2010),  1367--1396. 

\bibitem{FGR87} P. Frankl, R. L. Graham and V. R\"odl, \emph{On subsets of abelian groups with no $3$-term arithmetic progression}, J. Comb. Theory {\bf 45}, (1987), 157--161.



\bibitem{GT08} B. Green and T. Tao, \emph{The primes contain 
arbitrarily long arithmetic progressions}, Annals  Math. {\bf  167}, 
(2008) , 481--547.

\bibitem{wG08} W. T. Gowers, \emph{Quasirandom groups},
Combin. Probab. Comput. {\bf 17} (2008), 363--387.

\bibitem{Halmos} P. R. Halmos, \emph{Measure theory}, Graduate Texts in Mathematics {\bf 18} (1974), Springer, N. Y., xii+304 pp. ISBN 978-1-4684-9442-6. 

\bibitem{kH16} K. Henriot, \emph{Arithmetic progressions in sets 
	of small doubling}, Mathematika {\bf 62}, (2016), 587--613.


\bibitem{nH74} N. Hindman, \emph{Finite sums from sequences within cells of a partition of $N$}, J. Comb. Theory Ser. A {\bf 17}, (1974), 1--1.

\bibitem{eH12} E. Hrushovski, \emph{Stable group theory and approximate
	subgroups}, J. AMS {\bf 25}, (2012), 189--243.

\bibitem{eH13} E. Hrushovski, \emph{Approximate equivalence relations}, preprint, (2013), \url{http://www.ma.huji.ac.il/~ehud/approx-eq.pdf} 

\bibitem{HPS13} E. Hrushovski, A. Pillay and P. Simon, \emph{Generically stable and smooth
measures in NIP theories}, Trans. Amer. Math. Soc. {\bf 365} (2013),  2341--2366.

    \bibitem{jK87} H. J. Keisler, \emph{Measures and forking}, Annals
      Pure Appl. Logic {\bf 34}, (1987), 119--169.

\bibitem{MPP20} A. Martin-Pizarro and D. Palac\'in, \emph{Stabilizers, Measures 
and IP-sets}, to appear in Research Trends in Contemporary Logic, Melvin 
Fitting et al. (eds.), College Publications, (2020), 
  \url{http://arxiv.org/abs/1912.07252}
  
\bibitem{MPP04} A. Martin-Pizarro and A. Pillay, \emph{Elliptic and 
hyperelliptic 
curves over supersimple fields}, J. London Math. Soc. {\bf 69} 
(2004), 1--13.

\bibitem{MW15} J.-C. Massicot  and  F. O. Wagner, 
\emph{Approximate 
subgroups}, J. \'Ecole Polytechnique {\bf 2}, (2015),  55--63.

\bibitem{MOS18} S. Montenegro, A. Onshuus and P. Simon, \emph{Groups
    with f-generics in NTP$2$ and PRC fields}, 
  J. Inst. Math. Jussieu {\bf 19}, (2018), 821--853.

\bibitem{iN64} I. Namioka, \emph{F\o lner's conditions for amenable semi-groups}, Math. Scand. {\bf 15} (1964), 18--28.

\bibitem{dP20} D. Palac\'in, \emph{On compactifications and
    product-free sets}, J.  London Math. Soc. {\bf 101}, (2020), 
    156--174.

\bibitem{sP18} S. Peluse, \emph{Mixing for three-term 
progressions in finite simple groups},  Math. Proc. Cambridge Phil. 
Soc. {\bf  165}, (2018), 279--286. 

\bibitem{PSW98} A. Pillay, T. Scanlon and F. O. Wagner,
  \emph{Supersimple fields and division rings}, Math. Res. Letters {\bf
    5}, (1998), 473--483.

\bibitem{kR53} K. Roth, \emph{On certain sets of integers}, 
J. London Math. Soc. {\bf 28}, (1953), 104--109.


\bibitem{tS09} T. Sanders, \emph{Three-term arithmetic 
progressions and sumsets}, Proc.  Edinburgh Math. Soc. {\bf  52},  
(2009), 211--233. 

\bibitem{tS17} T. Sanders, \emph{Solving xz=yy in certain subsets of finite groups}, Q. J. Math. {\bf 68}, (2017), 243--273.

\bibitem{tS19}  T. Sanders, \emph{Schur's colouring theorem for non-commuting 
pairs}, B. Austin Math. Soc. {\bf 100}, (2019), 446--452.

\bibitem{iS16} I. Schur, \emph{\"Uber die Kongruenz $x^m+y^m=z^m (\textrm{mod.} 
p)$},  
Jahresbericht DMV {\bf 25}, (1916), 114--117.

\bibitem{pSbook} P. Simon, \emph{A Guide to NIP Theories}, Lecture Notes in Logic (2015),  Cambridge University Press, Cambridge, U.K. vii+156 pp. ISBN 978-1-107-05775-3. 


\bibitem{dS12} D. Simmons, \emph{Conditional measures and conditional expectation; Rohlin’s disintegration theorem}, Discrete Contin. Dyn. Syst. {\bf 32}, (2012), 2565--2582 (2012).

\bibitem{sS17} S. Starchenko, \emph{N{IP}, {K}eisler measures 
and
	combinatorics}, in S\'{e}minaire Bourbaki, Ast\'{e}risque {\bf
	390}, (2017), 303--334.

\bibitem{eS75} E. Szemer\'edi, \emph{On sets of integers 
containing no $k$ elements in arithmetic progression}, 
Acta Arith. {\bf 27}, (1975), 199--245.

\bibitem{tT08} T. Tao, \emph{Product set estimates for non-commutative groups}, 
Combinatorica {\bf 28}, (2008), 547--594.


\bibitem{tT13} T. Tao, \emph{Mixing for progressions in 
nonabelian groups}, Forum of Math., Sigma {\bf 1}, (2013),  
\url{doi:10.1017/fms.2013.2}

\bibitem{TW19} C. Terry and J. Wolf, \emph{Stable arithmetic regularity in the finite field model}, Bull. London Math. Soc., {\bf 51}, (2019), 70--88.

\end{thebibliography}
\end{document}